\documentclass[final,leqno]{siamltex}

\usepackage{amsfonts,amssymb} 
\usepackage{amsmath} 
\usepackage{enumerate,subfigure}
\usepackage{graphicx} 
\usepackage[colorlinks]{hyperref}
\usepackage{hypernat}

\graphicspath{{Figs/}}

\newtheorem{remark}{Remark}[section] 
\newtheorem{example}{Example}[section] 

\makeatletter

\newcommand{\Rmnum}[1]{\expandafter\@slowromancap\romannumeral #1@}
\makeatother

\title{A fourth-order maximum principle preserving operator splitting scheme for three-dimensional fractional Allen-Cahn equations \thanks{This research was supported by the Natural Science Foundation of China (Nos. 11402174, 41474103), the President's Fund-Research Start-up Fund from the Chinese University of Hong Kong, Shenzhen, the Excellent Youth Foundation of Hunan Province of China (No. 2018JJ1042) and the Innovation-Driven Project of Central South University (No. 2018CX042).}}

\author{Dongdong He\footnotemark[4] \and
Kejia Pan\footnotemark[2] \and Hongling Hu\footnotemark[3] }
\begin{document}
\maketitle

\renewcommand{\thefootnote}{\fnsymbol{footnote}}

\footnotetext[4]{School of Science and Engineering, The Chinese University of Hong Kong, Shenzhen, Guangdong 518172, China (hedongdong@cuhk.edu.cn)}
\footnotetext[2]{Corresponding author. School of Mathematics and Statistics, Central South University, Changsha, Hunan 410083, China (pankejia@hotmail.com)}
\footnotetext[3]{School of Mathematics and Statistics, Key Laboratory of High Performance Computing and Stochastic Information Processing, Hunan Normal
University, Changsha, Hunan 410081, China(honglinghu@hunnu.edu.cn)}

\renewcommand{\thefootnote}{\arabic{footnote}}

\begin{abstract}
In this paper, by using Strang's  second-order splitting method, the numerical procedure for the three-dimensional (3D)  space fractional  Allen-Cahn  equation can be divided into three steps. The first and third steps involve an ordinary differential equation, which can be solved analytically. The intermediate step involves a 3D linear fractional diffusion equation, which is solved by the Crank-Nicolson alternating directional implicit (ADI) method. The ADI technique can convert  the multidimensional problem into a series of one-dimensional problems, which greatly reduces the computational cost. A fourth-order difference scheme is adopted for discretization of the space fractional derivatives. Finally,  Richardson extrapolation is exploited to increase the temporal accuracy. The proposed method is shown to be unconditionally stable by Fourier analysis. Another contribution of this paper is to show that the numerical solutions satisfy the discrete maximum principle under reasonable time step constraint. For fabricated smooth solutions, numerical results show that the proposed method is unconditionally stable and fourth-order accurate in both time and space variables. In addition, the discrete  maximum principle is also numerically verified.
\end{abstract}



\pagestyle{myheadings}
\thispagestyle{plain}
\markboth{D.D. HE AND K.J. PAN}{An operator splitting scheme for 3D fractional Allen-Cahn equations}

\section{Introduction}
In this paper, we investigate the numerical solution of  the following space fractional Allen-Cahn equation~\cite{Allen,Chan2017,Gui2015,Nec2008}

\begin{align}\label{AC}
u_t=\varepsilon^2L_{\alpha}u-f(u), \quad \mathbf{x}\in \Omega, \quad  t \in (0,T],
\end{align}
with initial condition
\begin{align}\label{IC}
u(\mathbf{x},0)=u_0(x), \quad \mathbf{x}\in \Omega,
\end{align}
and the homogeneous Dirichlet boundary condition
\begin{align}\label{BC}
u(\mathbf{x},t)=0, \quad \mathbf{x} \ \rm{on}\ \partial\Omega, \quad  t \in [0,T],
\end{align}
where $\Omega$ is a rectangular domain, $\Omega=[0,1]^2$ in two dimension and  $\Omega=[0,1]^3$  in three dimension, $\alpha \in (1,2]$, and the nonlinear term $f(u)$ is taken as the polynomial double-well potential
\begin{align}\label{DW}
f(u)=u^3-u.
\end{align}
Here, the fractional Laplacian operator $L_{\alpha}$ replaces the standard Laplacian operator. In one dimension, the fractional Laplacian $L_{\alpha}u\ (\alpha \in (1,2))$ for $u$ defined in the interval $x\in [a,b]$ with homogeneous Dirichlet boundary condition is given as follows,
\begin{align}\label{1Dfr}
\mathcal{L}_{\alpha}u=\mathcal{L}^{\alpha}_xu:=\frac{1}{-2\cos(\frac{\alpha\pi}{2})}\left(_{a}D^{\alpha}_xu+_{x}D^{\alpha}_bu\right),
\end{align}
where the left and right Riemann-Liouville fractional derivatives are respectively defined as
$$_{a}D^{\alpha}_xu=\frac{1}{\Gamma(2-\alpha)}\frac{d^2}{dx^2}\int^x_a\frac{u(\xi)}{(x-\xi)^{\alpha-1}}d\xi,$$
$$_{x}D^{\alpha}_bu=\frac{1}{\Gamma(2-\alpha)}\frac{d^2}{dx^2}\int^b_x\frac{u(\xi)}{(\xi-x)^{\alpha-1}}d\xi.$$
The fractional operators in two dimension and three dimension  can be defined in a similar way, for example, the 3D fractional Laplacian $\mathcal{L}_{\alpha}u$ is defined as
$$\mathcal{L}_{\alpha}u=\mathcal{L}^{\alpha}_xu+\mathcal{L}^{\alpha}_yu+\mathcal{L}^{\alpha}_zu.$$
For the case $\alpha=2$, the above fractional  Allen-Cahn equation  will be reduced into the standard Allen-Cahn equation, which is widely studied in the literature.
When neglecting nonlinear term $f(u)$, equation (\ref{AC}) reduces to the fractional diffusion equation, which has been  numerically studied extensively in recent years ~\cite{Lin2017,Pan2014,Pang2012,Zhang2015}.

The fractional Allen-Cahn equation can be viewed as the $L^2$-gradient flow of the following fractional analogue Ginzburg-Landau free energy functional
\begin{align}
\mathcal{E}(u)=\int_{\Omega}F(u)-\frac{1}{2}\varepsilon^2uL_{\alpha}u du,
\end{align}
with $F(u)=\frac{1}{4}(u^2-1)^2$.


 The standard Allen-Cahn equation~\cite{Allen} was first introduced to describe the motion of anti-phase boundaries in crystalline solids. It can be considered as the $L^2$-gradient flow of the Ginzburg-Landau free energy. Recently, the Allen-Cahn equation, regarded as one of the diffusion-interface phase field  models, has been widely applied to many complicated moving interface problems, for example, vesicle membranes, the nucleation of solids and the mixture of two incompressible fluids~\cite{Du2004,Evans1991,Evans1992,Liu2003,Yang2006,Yue2005,Yue2006}. It is known that the Allen-Cahn equation has two intrinsic properties: one is the energy decreasing property, the other is the maximum principle~\cite{Evans1992}.  Since the Allen-Cahn equation is a nonlinear partial differential equation, the exact solution is not available. Numerical computations are essentially important to understand the behavior of the solution. In the existing literature, the Allen-Cahn equation was numerically extensively studied~\cite{Choi2009, Shen2010,Shen2016,Tang2016,LeeCMA2014,Leephya2015,YangX2003,YangX20131,YangX2009,YangX20132,Zhang2009}. For example, Choi {\it et al.}~\cite{Choi2009} proposed an unconditionally gradient stable nonlinear scheme with both discrete maximum principle and energy decreasing properties. And the discrete energy stability can be found in~\cite{Shen2010, YangX2003,YangX20131,YangX20132,Zhang2009}. More recently, the discrete maximum principle is discussed in~\cite{Hou2017,Shen2016,Tang2016}. However,
 the implicit-explicit scheme proposed in~\cite{Tang2016} is only first-order accurate in time variable, the fully discretized Crank-Nicolson scheme proposed in \cite{Hou2017} is a nonlinear scheme, and the authors pointed out that it is still remains open to see whether the maximum principle is still true for high-order accurate linear schemes, which is a difficult issue~\cite{Tang2016,Hou2017}.  Besides, the Allen-Cahn equation was also numerically investigated by using the operator splitting scheme~\cite{LeeCMA2014,Leephya2015,YangX2009}.

Recently, fractional differential equations have attracted many attention.  For time-fractional diffusion equations, finite difference methods and spectral methods have been used to investigate the solutions~\cite{Langlands,Li2009,Lin2007,Yuste2005}. For space fractional differential equations, there are also quite a lot of numerical studies for different equations~\cite{Yang2010,Wang20161,Wang20162,Liu2005,Tian2015,He,Bueno2014,Burrage2012,Zhuang2009,Hou2017,Song2016,Zhai2016,Saadatmandi2011,Moghaderi2017,Dehghan2017,Saadatmandi2011}.
 In particular, for the above space fractional  Allen-Cahn equation (\ref{AC}), Bueno-Orovio {\it et al.}~\cite{Bueno2014} used an implicit finite element method, Burrage {\it et al.}~\cite{Burrage2012} used a Fourier spectral method, and Hou {\it et al.}~\cite{Hou2017} used a finite difference method.   One of intrinsic property of the above space fractional  Allen-Cahn equation (\ref{AC}) is the maximum principle, which says that the value of the solution $u(\mathbf{x},t)$ is bounded by $1$ for any time $t>0$, provided the initial value $u_0(\mathbf{x})$ is bounded by $1$. Although the method of~\cite{Shen2016,Tang2016} can be extended to the above space fractional  Allen-Cahn equation (\ref{AC}) with the discrete maximum principle, this method has only a first-order accuracy in time variable. Recently, Hou {\it et al.}~\cite{Hou2017} proposed a finite difference scheme with discrete maximum principle, where the method is second-order accurate both in time and space variables. However, due to the nonlinear nature of the scheme and no dimensional splitting techniques,  the method of~\cite{Hou2017} will generally be computationally expensive especially when solving 3D problems.
And Song {\it et al.}~\cite{Song2016} first proposed a ADI method for the 2D space fractional Allen-Cahn equation and applied the proposed method to simulate the incompressible two-phase flows by coupling this fractional Allen-Cahn equation and Navier-Stokes equations, an extra linear term is added into the system in order to obtain the unconditional stability. Similar treatments by introducing an extra linear stabilized term for the integer order Allen-Cahn equation can be found in~\cite{YangX20131,Zhai2014,Shen2016,Tang2016}. As far as we aware, there is no study on high-order maximum principle preserving schemes for the space fractional Allen-Cahn equation, which is pointed as an open problem even in the case of the integer order Allen-Cahn equation~\cite{Tang2016}.

In this paper, we will develop a fourth-order maximum principle preserving operator splitting method for solving the 2D and 3D fractional Allen-Cahn equations (\ref{AC}).  First, by using a second-order operator splitting method, the numerical solution of the fractional Allen-Cahn equation can be obtained by three steps. The first and third steps involve   an ordinary differential equation (ODE), which can be solved analytically. The intermediate step involves a 2D/3D fractional diffusion equation, Crank-Nicolson scheme is adopted for time discretization, and the ADI method~\cite{Peaceman,HSun,He2017,Chen2018,He2018} combined with a fourth-order difference scheme is used for spatial discretization. The ADI technique converts the multidimensional diffusion problem into a series of one-dimensional problems, which greatly reduces the computational cost. Finally, Richardson extrapolation is exploited to increase the temporal accuracy to fourth order. The proposed method does not introduce any extra stabilized term and is shown to be unconditionally stable by the Von Neumann stability analysis for second step and a simple analysis for first and third steps. Another contribution of this paper is to show that the numerical solution satisfies the discrete maximum principle under reasonable time step constraint. Numerical experiments are carried out for both 2D and 3D space fractional Allen-Cahn equations.  For fabricated smooth solutions, results confirm  that the proposed method is unconditionally stable and fourth-order accurate for both time and space variables. Moreover,  the discrete  maximum principle is well verified numerically.

The rest of this paper is organized as follows. Section~\ref{sec2} provides the operator splitting method for the 3D fractional  Allen-Cahn equation. Section~\ref{sec3} proves unconditional stability of the proposed method. The discrete maximum principle is obtained in Section~\ref{sec4}. Section~\ref{sec5} presents the numerical results which confirm the theoretical results. And the conclusion is given in final section.

\section{Numerical method}\label{sec2}
In the following, we will present the numerical method to sovle the 3D fractional  Allen-Cahn equation, the proposed method can be straightforwardly applied to solve the 2D fractional  Allen-Cahn equation.

For a positive integer $N$, let $\Delta t= T/N$, $t_n=n\Delta t$. The time domain $[0,T]$ is covered by $\{t_n\}$. Let $v^n$ be the approximation of $v(x,y,z,n\Delta t)$ for an arbitary function $v(x,y,z,t)$.
The solution domain is defined as $\Omega\times [0,T]$ ($\Omega=[0,1]^3$), which is covered by a uniform grid
$\Omega_h=\{(x_i,y_j,z_k,t_n)|x_i=ih_x,y_j=jh_y,z_k=kh_z,t_n=n\Delta t, i=0,\cdots,M_x, j=0,\cdots, M_y, k=0,\cdots, M_z, n=0,\cdots, N\}$, where $h_x={1}/{M_x}, h_y={1}/{M_y}, h_z={1}/{M_z}$. Let $U^n=(U^n_{i,j,k})_{(M_x+1)\times(M_y+1)\times(M_z+1)}$ be the numerical solution at time level $t=t_n$, the homogeneous Dirichlet boundary condition (\ref{BC}) gives $$U^{n}_{0,j,k}=U^n_{M_x,j,k}=U^{n}_{i,0,k}=U^{n}_{i,M_y,k}=U^{n}_{i,j,0}=U^{n}_{i,j,M_z}=0$$ for any $n=0,\cdots,N$. And we denote
$$\|U^n\|_{\infty}=\max\limits_{\substack{1\leq i\leq M_x-1\\ 1\leq j\leq M_y-1 \\1\leq k \leq M_z-1}}|U^n_{i,j,k}|.$$

\subsection{Temporal  discretization}
Now we rewrite the fractional Allen-Cahn equation as the following evolution equation,
\begin{align}\label{Two}
\frac{\partial u}{\partial t}=\mathcal{L}_1u+\mathcal{L}_2u,
\end{align}
where the operators $\mathcal{L}_1, \mathcal{L}_2$ are defined as
$$\mathcal{L}_1u=-f(u)=u-u^3,\quad \mathcal{L}_2u=\varepsilon^2L_{\alpha}u.$$

According Strang's second-order splitting method~\cite{Strang}, the numerical solution of Eq. (\ref{Two}) in the time interval $[t_n,t_{n+1}]$ can be obtained as follows,
\begin{align}\label{Splitting}
U^{n+1}=\left(\mathcal{L}^{\frac{\Delta t}{2}}_1\circ \mathcal{L}^{\Delta t}_2\circ \mathcal{L}^{\frac{\Delta t}{2}}_1\right) U^n,
\end{align}
where $\mathcal{L}^{\Delta t}_1$ and $\mathcal{L}^{\Delta t}_2$ are the evolution operators for $\frac{\partial u}{\partial t}=\mathcal{L}_1u$ and $\frac{\partial u}{\partial t}=\mathcal{L}_2u$, respectively.

More precisely, we shall write the above splitting operator into three steps as follows,
\begin{align}
\frac{\partial \tilde{u}}{\partial t}&=-\frac{1}{2}f(\tilde{u})=\frac{1}{2}(\tilde{u}-\tilde{u}^3),\quad  \tilde{u}^n=U^n,\quad t\in [t_n, t_{n+1}], \label{stage1}\\
\frac{\partial \bar{u}}{\partial t}&=\varepsilon^2L_{\alpha}\bar{u}=\varepsilon^2\left(\mathcal{L}^{\alpha}_x\bar{u}+\mathcal{L}^{\alpha}_y\bar{u}+\mathcal{L}^{\alpha}_z\bar{u}\right),\quad  \bar{u}^n=\tilde{u}^{n+1},\quad t\in [t_n, t_{n+1}], \label{stage2}\\
\frac{\partial \hat{u}}{\partial t}&=-\frac{1}{2}f(\hat{u})=\frac{1}{2}(\hat{u}-\hat{u}^3),\quad  \hat{u}^n=\bar{u}^{n+1},\quad t\in [t_n, t_{n+1}]. \label{stage3}
\end{align}
The numerical solution at $t=t_{n+1}$ is given by $U^{n+1}=\hat{u}^{n+1}$.

The first and third steps (\ref{stage1}) and (\ref{stage3}) solve the same ODE, which can be calculated analytically~\cite{YangX2009,Leephya2015}, i.e.,
\begin{align}
\tilde{u}^{n+1}&=\frac{U^n}{\sqrt{(U^n)^2+\left(1-(U^n)^2\right)e^{-\Delta t}}},\label{analytical1}\\
\hat{u}^{n+1}&=\frac{\bar{u}^{n+1}}{\sqrt{(\bar{u}^{n+1})^2+\left(1-(\bar{u}^{n+1})^2\right)e^{-\Delta t}}}.\label{analytical2}
\end{align}

The intermediate step (\ref{stage2}) involves solving a 3D space fractional diffusion equation. A Crank-Nicolson ADI method is proposed for Eq. (\ref{stage2}), which will be described in the following.

We first apply the Crank-Nicolson scheme for temporal  discretization of Eq. (\ref{stage2}), i.e.,
\begin{align}\label{eq1}
\frac{\bar{u}^{n+1}-\bar{u}^n}{\Delta t}=\varepsilon^2\left(\mathcal{L}^{\alpha}_x+\mathcal{L}^{\alpha}_y+\mathcal{L}^{\alpha}_z\right)\frac{\bar{u}^{n+1}+\bar{u}^n}{2}+O(\Delta t^2),
\end{align}
Collecting the terms for $\bar{u}^{n+1}$ and $\bar{u}^{n}$ in (\ref{eq1}), one can get
\begin{align}\label{eq2}
\left(\frac{1}{\Delta t}-\frac{\varepsilon^2}{2}\left(\mathcal{L}^{\alpha}_x+\mathcal{L}^{\alpha}_y+\mathcal{L}^{\alpha}_z\right)\right)\bar{u}^{n+1}=\left(\frac{1}{\Delta t}+\frac{\varepsilon^2}{2}\left(\mathcal{L}^{\alpha}_x+\mathcal{L}^{\alpha}_y+\mathcal{L}^{\alpha}_z\right)\right)\bar{u}^{n}+O(\Delta t^2).
\end{align}
Eq. (\ref{eq2}) is equivalent to
\begin{align}\label{eq3}
&\frac{1}{\Delta t}\left(1-\frac{\Delta t\varepsilon^2}{2} \mathcal{L}^{\alpha}_x\right)\left(1-\frac{\Delta t\varepsilon^2}{2}\mathcal{L}^{\alpha}_y\right)\left(1-\frac{\Delta t\varepsilon^2}{2} \mathcal{L}^{\alpha}_z\right)\bar{u}^{n+1}\\
&=\frac{1}{\Delta t}\left(1+\frac{\Delta t\varepsilon^2}{2} \mathcal{L}^{\alpha}_x\right)\left(1+\frac{\Delta t\varepsilon^2}{2} \mathcal{L}^{\alpha}_y\right)\left(1+\frac{\Delta t\varepsilon^2}{2} \mathcal{L}^{\alpha}_z\right)\bar{u}^{n}\notag\\
&+\frac{\Delta t\varepsilon^4}{4}(\mathcal{L}^{\alpha}_x\mathcal{L}^{\alpha}_y+\mathcal{L}^{\alpha}_x\mathcal{L}^{\alpha}_z+\mathcal{L}^{\alpha}_y\mathcal{L}^{\alpha}_z)(\bar{u}^{n+1}-\bar{u}^{n})-\frac{\Delta t^2\varepsilon^6}{8}\mathcal{L}^{\alpha}_y\mathcal{L}^{\alpha}_x\mathcal{L}^{\alpha}_z(\bar{u}^{n+1}+\bar{u}^{n})+O(\Delta t^2).\notag
\end{align}
Since
\begin{align}
  &\frac{\Delta t\varepsilon^4}{4}(\mathcal{L}^{\alpha}_x\mathcal{L}^{\alpha}_y+\mathcal{L}^{\alpha}_x\mathcal{L}^{\alpha}_z+\mathcal{L}^{\alpha}_y\mathcal{L}^{\alpha}_z)(\bar{u}^{n+1}-\bar{u}^{n})\notag\\
  &=\frac{\Delta t\varepsilon^4}{4}(\mathcal{L}^{\alpha}_x\mathcal{L}^{\alpha}_y+\mathcal{L}^{\alpha}_x\mathcal{L}^{\alpha}_z+\mathcal{L}^{\alpha}_y\mathcal{L}^{\alpha}_z)\Big(\Delta t\frac{\partial \bar{u}^{n+\frac{1}{2}}}{\partial t}+O(\Delta t)^3\Big)=O(\Delta t^2), \notag
\end{align}
and
$$\frac{\Delta t^2\varepsilon^8}{8} \mathcal{L}^{\alpha}_x\mathcal{L}^{\alpha}_z\mathcal{L}^{\alpha}_z(u^{n+1}+u^{n})=\frac{\Delta t^2\varepsilon^8}{8} \mathcal{L}^{\alpha}_x\mathcal{L}^{\alpha}_y\mathcal{L}^{\alpha}_z\left(2\bar{u}^{n+\frac{1}{2}}+O(\Delta t^2)\right)=O(\Delta t^2),$$
Eq. (\ref{eq3}) is  indeed
\begin{align}\label{eq4}
&\frac{1}{\Delta t}\left(1-\frac{\Delta t\varepsilon^2}{2} \mathcal{L}^{\alpha}_x\right)\left(1-\frac{\Delta t\varepsilon^2}{2}\mathcal{L}^{\alpha}_y\right)\left(1-\frac{\Delta t\varepsilon^2}{2} \mathcal{L}^{\alpha}_z\right)\bar{u}^{n+1}\nonumber\\
=&\frac{1}{\Delta t}\left(1+\frac{\Delta t\varepsilon^2}{2} \mathcal{L}^{\alpha}_x\right)\left(1+\frac{\Delta t\varepsilon^2}{2} \mathcal{L}^{\alpha}_y\right)\left(1+\frac{\Delta t\varepsilon^2}{2} \mathcal{L}^{\alpha}_z\right)\bar{u}^{n}+O(\Delta t^2).
\end{align}
The above discretization is only for time.

\subsection{Spatial discretization}

For spatial discretization, with the homogenous boundary condition (\ref{BC}), the second-order difference scheme for the space fractional  derivatives is given as~\cite{Celik}
\begin{align}
\mathcal{L}^{\alpha}_xu\approx-\frac{1}{h^{\alpha}_1}\sum^{i-1}_{s=i-M_x+1}c^{\alpha}_{s}u_{i-s,j,k}=-\frac{1}{h^{\alpha}_1}\sum^{M_x-1}_{s=1}c^{\alpha}_{i-s}u_{s,j,k},\\
\mathcal{L}^{\alpha}_yu\approx-\frac{1}{h^{\alpha}_2}\sum^{j-1}_{s=j-M_y+1}c^{\alpha}_{s}u_{i,j-s,k}=-\frac{1}{h^{\alpha}_2}\sum^{M_y-1}_{s=1}c^{\alpha}_{j-s}u_{i,s,k},\\
\mathcal{L}^{\alpha}_zu\approx-\frac{1}{h^{\alpha}_3}\sum^{k-1}_{s=k-M_z+1}c^{\alpha}_{s}u_{i,j,k-s}=-\frac{1}{h^{\alpha}_3}\sum^{M_z-1}_{s=1}c^{\alpha}_{k-s}u_{i,j,s},
\end{align}
where
\begin{align}
c^{\alpha}_0&=\frac{\Gamma(\alpha+1)}{\left(\Gamma(\frac{\alpha}{2}+1)\right)^2}, \\
  c^{\alpha}_s&=\frac{(-1)^{s}\Gamma(\alpha+1)}{\Gamma(\frac{\alpha}{2}-s+1)\Gamma(\frac{\alpha}{2}+s+1)}=\Big(1-\frac{\alpha+1}{\frac{\alpha}{2}+s}\Big)c_{s-1}^\alpha, \quad \textrm{for } s\in \mathbb{Z}. \label{calpha}
\end{align}

Denote the operators $\Delta^{\alpha}_x, \Delta^{\alpha}_y, \Delta^{\alpha}_z$ and the identity operator $I$ as follows,
\begin{align}
[\Delta^{\alpha}_x{u}]_{i,j,k}&=-\sum^{i-1}_{s=i-M_x+1}c^{\alpha}_{s}{u}_{i-s,j,k}=-\sum^{M_x-1}_{s=1}c^{\alpha}_{i-s}{u}_{s,j,k},\label{sp1}\\
[\Delta^{\alpha}_y{u}]_{i,j,k}&=-\sum^{j-1}_{s=j-M_y+1}c^{\alpha}_{s}{u}_{i,j-s,k}=-\sum^{M_y-1}_{s=1}c^{\alpha}_{j-s}{u}_{i,s,k},\label{sp2}\\
[\Delta^{\alpha}_z{u}]_{i,j,k}&=-\sum^{k-1}_{s=k-M_z+1}c^{\alpha}_{s}{u}_{i,j,k-s}=-\sum^{M_z-1}_{s=1}c^{\alpha}_{k-s}{u}_{i,j,s},\label{sp3}
\end{align}
and
\begin{equation}
  [I{u}]_{i,j,k}={u}_{i,j,k}.\label{sp4}
\end{equation}

Let $\mathcal{A}^{\alpha}_{x}, \mathcal{A}^{\alpha}_{y}, \mathcal{A}^{\alpha}_{z}$ be the average operators defined as~\cite{Hao2016}
  \begin{align}
    \mathcal{A}^{\alpha}_{x} u_{i,j,k}&= \frac{\alpha}{24} u_{i-1,j,k} + \left(1-\frac{\alpha}{12}\right) u_{i,j,k} + \frac{\alpha}{24} u_{i+1,j,k},\\
    \mathcal{A}^{\alpha}_{y} u_{i,j,k}&= \frac{\alpha}{24} u_{i,j-1,k} + \left(1-\frac{\alpha}{12}\right) u_{i,j,k} + \frac{\alpha}{24} u_{i,j+1,k},\\
      \mathcal{A}^{\alpha}_{z} u_{i,j,k}&= \frac{\alpha}{24} u_{i,j,k-1} + \left(1-\frac{\alpha}{12}\right) u_{i,j,k} + \frac{\alpha}{24} u_{i,j,k+1}.
      \end{align}
The fourth-order  difference scheme for the space fractional  derivatives is given as~\cite{Hao2016,He}
\begin{align}
(\mathcal{L}^{\alpha}_xu)_{i,j,k}&=\frac{1}{h^{\alpha}_1}[(\mathcal{A}^{\alpha}_{x})^{-1}\Delta^{\alpha}_x{u}]_{i,j,k} +O(h^4_1),\label{eqa}\\
(\mathcal{L}^{\alpha}_yu)_{i,j,k}&=\frac{1}{h^{\alpha}_2}[(\mathcal{A}^{\alpha}_{y})^{-1}\Delta^{\alpha}_y{u}]_{i,j,k} +O(h^4_2),\\
(\mathcal{L}^{\alpha}_yu)_{i,j,k}&=\frac{1}{h^{\alpha}_3}[(\mathcal{A}^{\alpha}_{z})^{-1}\Delta^{\alpha}_z{u}]_{i,j,k}  +O(h^4_3).\label{eqb}
\end{align}

Substituting (\ref{eqa})-(\ref{eqb}) into (\ref{eq4}) and evaluating at ($x_i, y_j,z_k$), we have
\begin{align}\label{eqc}
&\frac{1}{\Delta t}\left[\Big(I-\frac{\Delta t\varepsilon^2}{2h^{\alpha}_1} (\mathcal{A}^{\alpha}_{x})^{-1}\Delta^{\alpha}_x\Big)\Big(I-\frac{\Delta t\varepsilon^2}{2h^{\alpha}_2} (\mathcal{A}^{\alpha}_{y})^{-1}\Delta^{\alpha}_y\Big)
\Big(I-\frac{\Delta t\varepsilon^2}{2h^{\alpha}_3} (\mathcal{A}^{\alpha}_{z})^{-1}\Delta^{\alpha}_z\Big)
\bar{u}^{n+1}\right]_{i,j,k}\\
&=\frac{1}{\Delta t}\left[\Big(I+\frac{\Delta t\varepsilon^2}{2h^{\alpha}_1} (\mathcal{A}^{\alpha}_{x})^{-1}\Delta^{\alpha}_x\Big)\Big(I+\frac{\Delta t\varepsilon^2}{2h^{\alpha}_2} (\mathcal{A}^{\alpha}_{y})^{-1}\Delta^{\alpha}_y\Big)
\Big(I+\frac{\Delta t\varepsilon^2}{2h^{\alpha}_3} (\mathcal{A}^{\alpha}_{z})^{-1}\Delta^{\alpha}_z\Big)\bar{u}^{n}\right]_{i,j,k}\notag\\
&+O(\Delta t^2+h_x^4+h_y^4+h_z^4).\notag
\end{align}

Neglecting the truncation errors in (\ref{eqc}), applying the operator $\mathcal{A}^{\alpha}_{x}\mathcal{A}^{\alpha}_{y}\mathcal{A}^{\alpha}_{z}$ to both sides and introducing the intermediate  variable $u^*, u^{**}$, we obtain the D'Yakonov ADI-like scheme~\cite{Peaceman} as follows,
\begin{align}
\big[(\mathcal{A}^{\alpha}_{x}- \beta_x\Delta^{\alpha}_x)\bar{u}^{*}\big]_{i,j,k}&=\big[(\mathcal{A}^{\alpha}_{x}+\beta_x\Delta^{\alpha}_x)(\mathcal{A}^{\alpha}_{y}+\beta_y \Delta^{\alpha}_y)(\mathcal{A}^{\alpha}_{z}+\beta_z\Delta^{\alpha}_z)\bar{u}^{n}\big]_{i,j,k}, \label{ADI1} \\
\big[(\mathcal{A}^{\alpha}_{y}-\beta_y \Delta^{\alpha}_y)\bar{u}^{**}\big]_{i,j,k}&=\bar{u}^{*}_{i,j,k}, \label{ADI2} \\
\big[(\mathcal{A}^{\alpha}_{z}-\beta_z\Delta^{\alpha}_z)\bar{u}^{n+1}\big]_{i,j,k}&=\bar{u}^{**}_{i,j,k}, \label{ADI3}
\end{align}
where $\beta_x={\Delta t\varepsilon^2}/(2h^{\alpha}_1), \beta_y={\Delta t\varepsilon^2}/(2h^{\alpha}_2),\beta_z={\Delta t\varepsilon^2}/(2h^{\alpha}_3)$. Each of the above three equations is a one-dimensional linear system and all coefficient matrices are constant matrices whose inverse only need to be computed once during the whole computation. Thus, the above ADI method can be solved very efficiently.

\begin{remark}
  From the analytical solution (\ref{analytical2}) and homogeneous boundary condition (\ref{BC}), one can show that $\bar{u}^{n+1}$ also satisfies the homogeneous boundary condition. Thus, from above ADI scheme, one can see that $\bar{u}^{**}$ satisfies the  homogeneous boundary condition in $x,y$ directions and $\bar{u}^{*}$ satisfies the  homogeneous boundary condition in $x$ direction. These conditions are needed in the implementation of the above ADI scheme.
\end{remark}

\begin{remark}
  The above Crank-Nicolson ADI scheme (\ref{ADI1})-(\ref{ADI3}) is second-order accurate in time and fourth-order accurate in space. Replacing the average operator $\mathcal{A}^{\alpha}_{x}$, $\mathcal{A}^{\alpha}_{y}$ and $\mathcal{A}^{\alpha}_{z}$ by the identity operator defined in (\ref{sp4}), yields the following second-order scheme:
  \begin{align}
\big[(\mathcal{I}- \beta_x\Delta^{\alpha}_x)\bar{u}^{*}\big]_{i,j,k}&=\big[(\mathcal{I}+\beta_x\Delta^{\alpha}_x)(\mathcal{I}+\beta_y \Delta^{\alpha}_y)(\mathcal{I}+\beta_z\Delta^{\alpha}_z)\bar{u}^{n}\big]_{i,j,k}, \label{ADI11} \\
\big[(\mathcal{I}-\beta_y \Delta^{\alpha}_y)\bar{u}^{**}\big]_{i,j,k}&=\bar{u}^{*}_{i,j,k}, \label{ADI21} \\
\big[(\mathcal{I}-\beta_z\Delta^{\alpha}_z)\bar{u}^{n+1}\big]_{i,j,k}&=\bar{u}^{**}_{i,j,k}. \label{ADI31}
\end{align}
\end{remark}

\subsection{Richardson extrapolation}
The scheme (\ref{eqc}) is a higher-order perturbation of the Crank-Nicolson scheme (\ref{eq1}) in three space variables. Thus, the temporal order of accuracy of the  ADI scheme (\ref{ADI1})-(\ref{ADI3}) is two, which is the same as the  Strang's time splitting method (\ref{stage1})-(\ref{stage3}).
Consequently, the proposed operator splitting method (\ref{stage1})-(\ref{stage3}) together with the ADI scheme (\ref{ADI1})-(\ref{ADI3}) will be second-order accurate in time variable and fourth-order accurate in space variable.  In order to increase the time accuracy, we apply the following Richardson extrapolation for the final step numerical solution:
\begin{equation}\label{Richardson}
  \widetilde{U}^N(\Delta t, h_x,h_y,h_z) = \frac{4}{3}U^N(\Delta t, h_x,h_y,h_z) - \frac{1}{3} U^{N/2}(2\Delta t, h_x,h_y,h_z),
\end{equation}
where $U^N(\Delta t, h_x,h_y,h_z)$, $U^{{N}/{2}}(2\Delta t, h_x,h_y,h_z)$ are numerical solutions at the final step by using spatial meshsizes $h_x,h_y,h_z$ and
 time step $\Delta t$, $2\Delta t$, respectively.

If the exact solution has sufficient regularity, then the extrapolated solution $\widetilde{U}^N$  is fourth-order accurate in both time and space,
see last two columns in Tables \ref{table11}-\ref{table25} for details.

\section{Unconditional stability}\label{sec3}

In this section, we show that the first and third steps in the splitting method (\ref{stage1})-(\ref{stage3}) are unconditionally stable, and then use von Neumann linear stability analysis to prove the unconditional stability of the  ADI scheme (\ref{ADI1})-(\ref{ADI3}) for the second step under the condition that the solution $u$ is periodic and smooth.

\begin{lemma}\label{Lemma1}~\cite{Celik}
The coefficients $c^{\alpha}_s$  have the following properties for $-1<\alpha\leq 2$
\begin{align}
c^{\alpha}_0&=\frac{\Gamma(\alpha+1)}{\left(\Gamma(\frac{\alpha}{2}+1)\right)^2}>0,\nonumber \\
c^{\alpha}_p&=c^{\alpha}_{-p}\leq 0,\quad \textrm{for}\ p=\pm1, \pm 2, \cdots,\nonumber \\
\sum^{i-1}_{p=-M_{x(y,z)}+1+i,p\neq0}&|c_p^{\alpha}|< c_0^\alpha, \quad \textrm{for}\ i=1,  \cdots, M_{x(y,z)}-1.
\end{align}
 \end{lemma}

  \begin{lemma}\label{Lemma2}
Suppose that $z$ is a complex number and $a>0$ is a real number, then
\begin{align}
\left|\frac{a-z}{a+z}\right|\leq 1 \quad {\rm if\ and\ only\ if}\ {\rm Re(z)}\geq 0.
\end{align}
\end{lemma}
\begin{proof}
  This lemma is easy to verify.
\end{proof}

\begin{lemma}\label{theorem1}
At any time level $t=t_n$,  for any initial value $U^n_{i,j,k}(i=0,\cdots,M_x, j=0,\cdots,M_y,k=0,\cdots,M_z)$, the numerical solution $u^n_{i,j,k}$ given by (\ref{analytical1}) for the first step in the operator splitting method is unconditionally stable.
\end{lemma}
\begin{proof}
  For a given ($i,j,k$) ($i=0,\cdots,M_x, j=0,\cdots,M_y,k=0,\cdots,M_z$)\\
Case 1. if the component $U^n_{i,j,k}$ satisfies $|U^n_{i,j,k}|\leq 1$, then by using (\ref{analytical1}), one has $$|\tilde{u}^{n+1}_{i,j,k}|=\frac{|U^n_{i,j,k}|}{\sqrt{(U^n_{i,j,k})^2+\left(1-(U^n_{i,j,k})^2\right)e^{-\Delta t}}}\leq \frac{|U^n_{i,j,k}|}{\sqrt{(U^n_{i,j,k})^2}} = 1.$$
Case 2. if the component $U^n_{i,j,k}$ satisfies $|U^n_{i,j,k}|>1$, then by using (\ref{analytical1}), one has $$|\tilde{u}^{n+1}_{i,j,k}|=\frac{|U^n_{i,j,k}|}{\sqrt{(U^n_{i,j,k})^2(1-e^{-\Delta t})+e^{-\Delta t}}}\leq \frac{|U^n_{i,j,k}|}{\sqrt{(1-e^{-\Delta t})+e^{-\Delta t}}} = |U^n_{i,j,k}|.$$
Combining these two cases, one has
$$\|\tilde{u}^{n+1}\|_{\infty}\leq \max\left\{\|U^{n}\|_{\infty},1\right\},$$
which completes the proof of the lemma.
\end{proof}

\begin{lemma}\label{theorem2}
At any time level $t=t_n$,  for any initial value $\hat{u}^n_{i,j,k}(i=0,\cdots,M_x, j=0,\cdots,M_y,k=0,\cdots,M_z)$, the numerical solution $\hat{u}^{n+1}_{i,j,k}$ given by (\ref{analytical2}) for the third step in the operator splitting method is unconditionally stable.
\end{lemma}
\begin{proof}
  The proof of this lemma is similar as Lemma \ref{theorem1}.
\end{proof}

\begin{lemma}\label{theorem3}
 If $u$ is periodic in $x$, $y$ and $z$ directions and smooth, then the Crank-Nicolson ADI method (\ref{ADI1})-(\ref{ADI3}) is unconditionally stable.
\end{lemma}
\begin{proof}
   Let $\bar{u}^n_{i,j,k}$ be the numerical solution  of the Crank-Nicolson ADI method (\ref{ADI1})-(\ref{ADI3}). Since $u$ is periodic and smooth while the first step (\ref{stage1}) is an ODE, which is solved analytically.  Thus, we can assume that the numerical solution $\bar{u}^n_{i,j,k}$ at time level $t=t_n$ is also periodic, which  has the following form
\begin{align}\label{denotes}
&\bar{u}^{n}_{i,j,k}=\xi^n e^{I(w_xi+w_yj+w_zk)},
\end{align}
where $\xi^n$ is the amplitude at time level $n$, $I=\sqrt{-1}$ is the complex unit, and $w_x$  $w_y$ and $w_z$ are phase angles in $x$, $y$ and $z$ directions, respectively.

Substituting (\ref{denotes}) into (\ref{ADI1})-(\ref{ADI3}), one can get
\begin{align}\label{eigen}
\left|\frac{\bar{u}^{n+1}_{i,j,k}}{\bar{u}^{n}_{i,j,k}}\right|&=\left|\frac{1+\frac{\alpha(\cos w_x-1)}{12}-\beta_x\sum\limits^{i-1}_{s=i-M_x+1}c^{\alpha}_{s}e^{-Isw_x}}{1+\frac{\alpha(\cos w_x-1)}{12}+\beta_x\sum\limits^{i-1}_{s=i-M_x+1}c^{\alpha}_{s}e^{-Isw_x}}\right|\cdot\left|\frac{1+\frac{\alpha(\cos w_y-1)}{12}-\beta_y\sum\limits^{j-1}_{s=j-M_y+1}c^{\alpha}_{s}e^{-Isw_y}}{1+\frac{\alpha(\cos w_y-1)}{12}+\beta_y\sum\limits^{j-1}_{s=j-M_y+1}c^{\alpha}_{s}e^{-Isw_y}}\right|\\
&\cdot\left|\frac{1+\frac{\alpha(\cos w_z-1)}{12}-\beta_z\sum\limits^{k-1}_{s=k-M_z+1}c^{\alpha}_{s}e^{-Isw_z}}{1+\frac{\alpha(\cos w_z-1)}{12}+\beta_z\sum\limits^{k-1}_{s=k-M_z+1}c^{\alpha}_{s}e^{-Isw_z}}\right|.\notag
\end{align}

From Lemma~\ref{Lemma1}, we know that
$$\textrm{Re}\left(\sum\limits^{i-1}_{s=i-M_x+1}c^{\alpha}_{s}e^{-Isw_x}\right)=c^{\alpha}_0+\sum\limits^{i-1}_{s=i-M_x+1,s\neq0}c^{\alpha}_s\cos(sw_x)\geq c^{\alpha}_0-\sum\limits^{i-1}_{s=i-M_x+1,s\neq0}|c^{\alpha}_s|\geq0.$$

Since $1<\alpha\leq2$,
$$1+\frac{\alpha(\cos w_x-1)}{12}=\frac{12+\alpha(\cos w_x-1)}{12}\geq\frac{12-4}{12}>0.$$
By using Lemma~\ref{Lemma2}, one has
\begin{align}
\left|\frac{1+\frac{\alpha(\cos w_x-1)}{12}-\beta_x\sum\limits^{i-1}_{s=i-M_x+1}c^{\alpha}_{s}e^{-Isw_x}}{1+\frac{\alpha(\cos w_x-1)}{12}+\beta_x\sum\limits^{i-1}_{s=i-M_x+1}c^{\alpha}_{s}e^{-Isw_x}}\right|\leq 1.
\end{align}
Similarly, one has
\begin{align}
\left|\frac{1+\frac{\alpha(\cos w_y-1)}{12}-\beta_y\sum\limits^{j-1}_{s=j-M_y+1}c^{\alpha}_{s}e^{-Isw_y}}{1+\frac{\alpha(\cos w_y-1)}{12}+\beta_y\sum\limits^{j-1}_{s=j-M_y+1}c^{\alpha}_{s}e^{-Isw_y}}\right|\leq 1,\end{align}
 and
 \begin{align}
 \left|\frac{1+\frac{\alpha(\cos w_z-1)}{12}-\beta_z\sum\limits^{k-1}_{s=k-M_z+1}c^{\alpha}_{s}e^{-Isw_z}}{1+\frac{\alpha(\cos w_z-1)}{12}+\beta_z\sum\limits^{k-1}_{s=k-M_z+1}c^{\alpha}_{s}e^{-Isw_z}}\right|\leq 1.
\end{align}

Thus,
\begin{align}
\left|\frac{\bar{u}^{n+1}_{i,j,k}}{\bar{u}^{n}_{i,j,k}}\right|\leq 1.
\end{align}
This completes the proof of the lemma.
 \end{proof}

\begin{theorem}
 If $u$ is periodic and smooth, then the operator splitting scheme (\ref{stage1}), (\ref{ADI1})-(\ref{ADI3}) and (\ref{stage3})
 is unconditionally stable.
\end{theorem}
\begin{proof}
  Combining the results of Lemma \ref{theorem1}, \ref{theorem2}, \ref{theorem3} gives the theorem.
\end{proof}

\begin{remark}
Since the homogeneous Dirichlet boundary condition~(\ref{BC}) is used in this paper, the solution can always be extended into a periodic function with a convergent Fourier series expansion if the solution is smooth. Thus, for smooth solutions, the proposed operator splitting  method is unconditionally stable.
\end{remark}

\section{Discrete maximum principle}\label{sec4}
\begin{lemma}
 Let matrix $\mathbf{C}_x$ be defined as follows:
 \begin{align}
\mathbf{C}_x&=\begin{pmatrix}
-c^{\alpha}_0& 0&\cdots&\cdots&\cdots&0\\
-c^{\alpha}_{1}& -c^{\alpha}_0&-c^{\alpha}_{-1}&\cdots&\cdots&-c^{\alpha}_{-M+1}\\
-c^{\alpha}_{2}& -c^{\alpha}_{1}&-c^{\alpha}_0& -c^{\alpha}_{-1}&\cdots&-c^{\alpha}_{-M+2}\\
\vdots& \vdots&\ddots&\ddots&\ddots&\vdots\\
\vdots& \vdots&\cdots& -c^{\alpha}_{1}&-c^{\alpha}_0&-c^{\alpha}_{-1}\\
0& 0&\cdots&0&0&-c^{\alpha}_0\\
\end{pmatrix}_{(M_x+1)\times (M_x+1)}.
\end{align}
Then, $\mathbf{C}_x$ is strictly diagonally dominant~\cite{Celik}, i.e.
 \begin{align}\label{strict}
 |c_{ii}|=c^{\alpha}_0>\sum\limits_{j\neq i}|c_{ij}|, \quad \textrm{for}\ i=1,  \cdots, M_x+1.
 \end{align}
Similarly, we can define $(M_y+1)\times (M_y+1)$ square matrix $\mathbf{C}_y$ and $(M_z+1)\times (M_z+1)$ square matrix $\mathbf{C}_z$, which are also strictly diagonally dominant.
\end{lemma}

\begin{lemma}\label{Lemma3}\cite{Shen2016,Tang2016}
Let matrix $\mathbf{B}\in \mathbb{R}^{(M+1)\times(M+1)}$, $\mathbf{A}=a\mathbf{I}-\mathbf{B}$, where $a>0$, $\mathbf{I}$ is the identity matrix with same size of $\mathbf{B}$ and $\mathbf{B}$ is a negative
diagonally dominant matrix, i.e.
$$\forall\ i=1,\cdots, M+1,\quad  b_{ii}\leq0, \quad {\rm and}\quad b_{ii}+\sum_{j\neq i}|b_{ij}|\leq 0,$$
 then $\mathbf{A}$ is invertible and
 \begin{align}\label{inverse}
\|\mathbf{A}^{-1}\|_{\infty}\leq \frac{1}{a}.
 \end{align}
\end{lemma}

In the section, we will show that, under certain reasonable time step constraint, the discrete maximum principle for the proposed method is valid.
\begin{theorem}\label{th4}
Assume that the initial value $u_0(x)$ satisfies $\max\limits_{x\in\bar{\Omega}}|u_0(x)|\leq 1$, then the numerical solution $U^{n}_{i,j,k}$ of
(\ref{stage1}), (\ref{ADI1})-(\ref{ADI3}) and (\ref{stage3}) satisfies the discrete maximum principle, i.e.,  $\|U^n\|_{\infty}\leq 1$ for any $n=0,1,\cdots,N$
if the time step satisfies
 \begin{align}\label{condition}
\frac{\alpha+2}{12}\frac{\max(h^{\alpha}_1,h^{\alpha}_2,h^{\alpha}_3)}{\varepsilon^2c^{\alpha}_0} \leq \Delta t\leq \frac{12 - \alpha}{6} \frac{\min(h^{\alpha}_1,h^{\alpha}_2,h^{\alpha}_3)}{\varepsilon^2c^{\alpha}_0}.
  \end{align}
\end{theorem}
\begin{proof}
   We prove the theorem by mathematical induction. Obviously, $\|U^n\|_{\infty}\leq 1$ for $n=0$ since $U^0_{i,j,k}=u_0(x_i,y_j,z_k).$
Assume that $\|U^k\|_{\infty}\leq 1$ ($k\leq n$) is valid, we want to show that $\|U^{n+1}\|_{\infty}\leq 1$. From (\ref{analytical1}), one can easily obtain that
\begin{align}\label{condition1}
 \|\tilde{u}^{n+1}\|_{\infty}\leq 1.
\end{align}
Next, we look at the numerical solution $\bar{u}^{n+1}$ of (\ref{ADI1})-(\ref{ADI3}). Since $\bar{u}^{n}=\tilde{u}^{n+1}$, one has
\begin{align}\label{condition2}
 \|\bar{u}^{n}\|_{\infty}\leq 1.
\end{align}

Let $\bar{u}^n=(\bar{u}^n_{i,j,k})_{(M_x+1)\times (M_y+1)\times (M_z+1)}$ be a 3D matrix including the boundary points, which are zero values. Denfine matrix $\mathbf{D}_x$ as follows
 \begin{align}
\mathbf{D}_x&=\begin{pmatrix}
-2&0&0&\cdots&\cdots&0\\
1& -2&1&0&\cdots&0\\
0&1&-2&1&\ddots&0\\
\vdots&\ddots& \ddots&\ddots&\ddots&\vdots\\
0&\cdots&0&1&-2&1\\
0&\cdots&0&0&0&-2\\
\end{pmatrix}_{(M_x+1)\times(M_x+1)}.
\end{align}
Another two square matrices $\mathbf{D}_y$ and $\mathbf{D}_z$ can be defined similarly.

From the definition of the operators $\Delta^{\alpha}_x,\Delta^{\alpha}_y,\Delta^{\alpha}_z, I$ in (\ref{sp1})-(\ref{sp4})
and the zero value on boundary  points, one can see that the application of operator
$\mathcal{A}^{\alpha}_{z}+\beta_z\Delta^{\alpha}_z$ to $\bar{u}^n$ is equivalent to multiply
each vector in the third dimension of $\bar{u}^n$ by  the matrix $\mathbf{I}_z+\frac{\alpha}{24}\mathbf{D}_z+\beta_z\mathbf{C}_z$ ($\mathbf{I}_z$ is the identity matrix with size $(M_z+1)\times(M_z+1)$),
the application of operator $\mathcal{A}^{\alpha}_{y}+\beta_y\Delta^{\alpha}_y$ to $\bar{u}^n$ is equivalent to multiply
each vector in the second dimension of $\bar{u}^n$ by  the matrix $\mathbf{I}_y+\frac{\alpha}{24}\mathbf{D}_y+\beta_y\mathbf{C}_y$ ($\mathbf{I}_y$ is the identity matrix with size $(M_y+1)\times(M_y+1)$),
and the application of operator $\mathcal{A}^{\alpha}_{x}+\beta_x\Delta^{\alpha}_x$ to $\bar{u}^n$ is equivalent to multiply
each vector in the first dimension of $\bar{u}^n$ by  the matrix $\mathbf{I}_x+\frac{\alpha}{24}\mathbf{D}_x+\beta_x\mathbf{C}_x$ ($\mathbf{I}_x$ is the identity matrix with size $(M_x+1)\times(M_x+1)$).

In addition, (\ref{ADI1})-(\ref{ADI3}) is equivalent to
\begin{align}
\left[\left(\mathcal{A}^{\alpha}_{x}-\beta_x \Delta^{\alpha}_x\right)\left(\mathcal{A}^{\alpha}_{y}-\beta_y \Delta^{\alpha}_y\right)\left(\mathcal{A}^{\alpha}_{z}-\beta_z \Delta^{\alpha}_z\right)\bar{u}^{n+1}\right]_{i,j,k}=\left[\left(\mathcal{A}^{\alpha}_{x}+\beta_x \Delta^{\alpha}_x\right)\left(\mathcal{A}^{\alpha}_{y}+\beta_y \Delta^{\alpha}_y\right)\left(\mathcal{A}^{\alpha}_{z}+\beta_z \Delta^{\alpha}_z\right)\bar{u}^n\right]_{i,j,k},
\end{align}
which further yields
\begin{align}
\bar{u}_{i,j,k}^{n+1}=\left[\left(\mathcal{A}^{\alpha}_{z}-\beta_z \Delta^{\alpha}_z\right)^{-1}\left(\mathcal{A}^{\alpha}_{y}-\beta_y \Delta^{\alpha}_y\right)^{-1}\left(\mathcal{A}^{\alpha}_{x}-\beta_x \Delta^{\alpha}_x\right)^{-1}\left(\mathcal{A}^{\alpha}_{x}+\beta_x \Delta^{\alpha}_x\right)\left(\mathcal{A}^{\alpha}_{y}+\beta_y\Delta^{\alpha}_y\right)\left(\mathcal{A}^{\alpha}_{z}+\beta_z \Delta^{\alpha}_z\right)\bar{u}^n\right]_{i,j,k},
\end{align}
where the application of operator $(\mathcal{A}^{\alpha}_{z}-\beta_z\Delta^{\alpha}_z)^{-1}$
to a 3D matrix is equivalent to multiply  each vector in the third dimension of
 this 3D matrix by the matrix $(\mathbf{I}_z+\frac{\alpha}{24}\mathbf{D}_z-\beta_z\mathbf{C}_z)^{-1}$,
  the application of operator $(\mathcal{A}^{\alpha}_{y}-\beta_y\Delta^{\alpha}_y)^{-1}$ to a 3D matrix is equivalent to multiply
    each vector in the second dimension of this 3D matrix by the matrix $(\mathbf{I}_y+\frac{\alpha}{24}\mathbf{D}_y-\beta_y\mathbf{C}_y)^{-1}$,
     and the application of operator $(\mathcal{A}^{\alpha}_{x}-\beta_x\Delta^{\alpha}_x)^{-1}$ to a 3D matrix is equivalent to multiply
      each vector in the first dimension of this 3D matrix by the matrix $(\mathbf{I}_x+\frac{\alpha}{24}\mathbf{D}_x-\beta_x\mathbf{C}_x)^{-1}$.

Therefore, it is easy to check that $\bar{u}^{n+1}=(\bar{u}^{n+1}_{i,j,k})_{(M_x+1)\times (M_y+1)\times (M_z+1)}$ can be obtained from $\bar{u}^{n}=(\bar{u}^{n}_{i,j,k})_{(M_x+1)\times (M_y+1)\times (M_z+1)}$ through a series of one-dimensional vector transformations as follows:
\begin{enumerate}[(i)]
         \item Multiplying each vector in the third dimension of $\bar{u}^n$ by  the matrix $\mathbf{I}_z+\frac{\alpha}{24}\mathbf{D}_z+\beta_z\mathbf{C}_z$,
         \item Multiplying each vector in the second dimension of resulting  matrix in (i) by  the matrix $\mathbf{I}_y+\frac{\alpha}{24}\mathbf{D}_y+\beta_y \mathbf{C}_y$,
         \item Multiplying each vector in the first dimension of resulting  matrix in (ii) by  the matrix $\mathbf{I}_x+\frac{\alpha}{24}\mathbf{D}_x+\beta_x \mathbf{C}_x$,
         \item Multiplying each vector in the first dimension of resulting  matrix in (iii) by  the matrix $(\mathbf{I}_x+\frac{\alpha}{24}\mathbf{D}_x-\beta_x \mathbf{C}_x)^{-1}$,
         \item Multiplying each vector in the second dimension of resulting  matrix in (iv) by  the matrix $(\mathbf{I}_y+\frac{\alpha}{24}\mathbf{D}_y-\beta_y \mathbf{C}_y)^{-1}$,
         \item Multiplying each vector in the third dimension of resulting matrix in (v) by  the matrix $(\mathbf{I}_z+\frac{\alpha}{24}\mathbf{D}_z-\beta_z \mathbf{C}_z)^{-1}$.
       \end{enumerate}

If condition (\ref{condition}) is satisfied, then
 \begin{align}
1-\frac{\alpha}{12}-\beta_z c^{\alpha}_0>0,
  \end{align}
and
\begin{align}
\sum_{j}\left|\delta_{i,j}+\frac{\alpha}{24}d_{i,j}+\beta_z c_{i,j}\right|=1-\frac{\alpha}{12}-\beta_z c^{\alpha}_0<1, \quad {\rm for }\quad  i=1, M_z+1,
  \end{align}
 \begin{align}
\sum_{j}\left|\delta_{i,j}+\frac{\alpha}{24}d_{i,j}+\beta_z c_{i,j}\right|&\leq 1-\frac{\alpha}{12}-\beta_z c^{\alpha}_0+\frac{\alpha}{24}+\frac{\alpha}{24}+\beta_z \sum_{j\neq i}\left|c_{i,j}\right|\nonumber\\
&=1-\beta_z\left(c^{\alpha}_0-\sum_{j\neq i}|c_{i,j}|\right)\nonumber\\
&<1, \qquad {\rm for }\quad  i=2,\cdots, M_z.
  \end{align}
Thus,
 \begin{align}\label{condition3}
\left\|\mathbf{I}_z+\frac{\alpha}{24}\mathbf{D}_z+\beta_z\mathbf{C}_z\right\|_{\infty}< 1.
  \end{align}

By using (\ref{condition2}) and (\ref{condition3}), one gets
 \begin{align}
\left\|\left(\mathcal{A}^{\alpha}_{z}+\beta_z\Delta^{\alpha}_z\right)\bar{u}^n\right\|_{\infty}\leq \left\|\mathbf{I}_z+\frac{\alpha}{24}\mathbf{D}_z+\beta_z\mathbf{C}_z\right\|_{\infty}\cdot  \|\bar{u}^{n}\|_{\infty} \leq 1.
  \end{align}

Similarly,
\begin{align}
\left\|\mathbf{I}_y+\frac{\alpha}{24}\mathbf{D}_y+\beta_y\mathbf{C}_y\right\|_{\infty}< 1,\quad  \left\|\mathbf{I}_x+\frac{\alpha}{24}\mathbf{D}_x+\beta_x\mathbf{C}_x\right\|_{\infty}< 1.
\end{align}
Therefore,
 \begin{align}
\left\|\left(\mathcal{A}^{\alpha}_{y}+\beta_y\Delta^{\alpha}_y\right)\left(\mathcal{A}^{\alpha}_{z}+\beta_z\Delta^{\alpha}_z\right)\bar{u}^n\right\|_{\infty}\leq 1,
  \end{align}
and
 \begin{align}\label{middle}
\left\|\left(\mathcal{A}^{\alpha}_{x}+\beta_x\Delta^{\alpha}_x\right)\left(\mathcal{A}^{\alpha}_{y}+\beta_y\Delta^{\alpha}_y\right)\left(\mathcal{A}^{\alpha}_{z}+\beta_z \Delta^{\alpha}_z\right)\bar{u}^n\right\|_{\infty}\leq 1.
  \end{align}

If condition (\ref{condition}) is satisfied, then
\begin{align}
 \frac{\alpha}{12}-\beta_x c^{\alpha}_{0}&=\frac{\alpha}{12}-\frac{\Delta t\varepsilon^2}{2h^{\alpha}_1}c^{\alpha}_{0}\leq0, \\
 -\frac{\alpha}{24}-\beta_x c^{\alpha}_{1}&=-\frac{\alpha}{24}-\frac{\Delta t\varepsilon^2}{2h^{\alpha}_1}(1-\frac{\alpha+1}{\frac{\alpha}{2}+1})c^{\alpha}_0=-\frac{\alpha}{24}+\frac{\Delta t\varepsilon^2}{2h^{\alpha}_1}\frac{\alpha}{\alpha+2}c^{\alpha}_0\geq0,
\end{align}
where Eq.(\ref{calpha}) is used.

Now for  matrix $-\frac{\alpha}{24}\mathbf{D}_x+\beta_x\mathbf{C}_x$, one has
\begin{align}
\sum_{j\neq i}\left|-\frac{\alpha}{24}d_{i,j}+\beta_x c_{i,j}\right|=0\leq-\frac{\alpha}{12}+\beta_x c^{\alpha}_{0}=-\left(-\frac{\alpha}{24}d_{i,i}+\beta_x c_{i,i}\right), \quad {\rm for }\quad  i=1, M_x+1,
  \end{align}
  and
 \begin{align}
\sum_{j\neq i}\left|-\frac{\alpha}{24}d_{i,j}+\beta_x c_{i,j}\right|&=\sum_{j\neq {i,i\pm1}}\left|-\frac{\alpha}{24}d_{i,j}+\beta_x c_{i,j}\right|+(-\frac{\alpha}{24}-\beta_x c^{\alpha}_{1})+(-\frac{\alpha}{24}-\beta_x c^{\alpha}_{-1}),\nonumber\\
&=-\frac{\alpha}{12}+\sum_{j\neq {i,i\pm1}}\beta_x\left| c_{i,j}\right|+(-\beta_x c^{\alpha}_{1})+(-\beta_x c^{\alpha}_{-1})\nonumber\\
&=-\frac{\alpha}{12}+\sum_{j\neq {i}}\beta_x\left| c_{i,j}\right|\nonumber\\
&\leq-\frac{\alpha}{12}+\beta_x\left| c_{i,i}\right|\nonumber\\
&=-\frac{\alpha}{12}+\beta_x c^{\alpha}_0\nonumber\\
&=-\left(-\frac{\alpha}{24}d_{i,i}+\beta_x c_{i,i}\right),  \qquad {\rm for }\quad i=2,\cdots, M_x.
  \end{align}
  Thus,  $-\frac{\alpha}{24}\mathbf{D}_x+\beta_x\mathbf{C}_x$ is a negative diagonally dominant matrix. By using Lemma~\ref{Lemma3}, one has
 \begin{align}\label{condition4}
\left\|\left(\mathbf{I}_x+\frac{\alpha}{24}\mathbf{D}_x-\beta_x\mathbf{C}_x\right)^{-1}\right\|_{\infty}\leq 1.
  \end{align}

Applying the above condition and using (\ref{middle}), one has
\begin{align}
\left\|\left(\mathcal{A}^{\alpha}_{x}-\beta_x\Delta^{\alpha}_x\right)^{-1}\left(\mathcal{A}^{\alpha}_{x}+\beta_x\Delta^{\alpha}_x\right)\left(\mathcal{A}^{\alpha}_{y}+\beta_y \Delta^{\alpha}_y\right)\left(\mathcal{A}^{\alpha}_{z}+\beta_z\Delta^{\alpha}_z\right)\bar{u}^n\right\|_{\infty}\leq 1,
  \end{align}

Similarly,  if condition (\ref{condition}) is satisfied, one has
 \begin{align}
\left\|\left(\mathbf{I}_y+\frac{\alpha}{24}\mathbf{D}_y-\beta_y\mathbf{C}_y\right)^{-1}\right\|_{\infty}\leq 1,
\quad
\left\|\left(\mathbf{I}_z+\frac{\alpha}{24}\mathbf{D}_z-\beta_z\mathbf{C}_z\right)^{-1}\right\|_{\infty}\leq 1.
  \end{align}
Thus,
\begin{align}
\left\|\left(\mathcal{A}^{\alpha}_{y}-\beta_y\Delta^{\alpha}_y\right)^{-1}\left(\mathcal{A}^{\alpha}_{x}-\beta_x\Delta^{\alpha}_x\right)^{-1}\left(\mathcal{A}^{\alpha}_{x}+\beta_x \Delta^{\alpha}_x\right)\left(\mathcal{A}^{\alpha}_{y}+\beta_y\Delta^{\alpha}_y\right)\left(\mathcal{A}^{\alpha}_{z}+\beta_z \Delta^{\alpha}_z\right)\bar{u}^n\right\|_{\infty}\leq 1,
  \end{align}
  and
\begin{align}
\left\|\left(\mathcal{A}^{\alpha}_{z}-\beta_z\Delta^{\alpha}_z\right)^{-1}\left(\mathcal{A}^{\alpha}_{y}-\beta_y\Delta^{\alpha}_y\right)^{-1}\left(\mathcal{A}^{\alpha}_{x}-\beta_x \Delta^{\alpha}_x\right)^{-1}\left(\mathcal{A}^{\alpha}_{x}+\beta_x\Delta^{\alpha}_x\right)\left(\mathcal{A}^{\alpha}_{y}+\beta_y\Delta^{\alpha}_y\right)\left(\mathcal{A}^{\alpha}_{z}+\beta_z \Delta^{\alpha}_z\right)\bar{u}^n\right\|_{\infty}\leq 1.
  \end{align}

  This yields
\begin{align}
\left\|\bar{u}^{n+1}\right\|_{\infty}\leq 1.
  \end{align}

  From (\ref{analytical2}), one can obtain that
\begin{align}
 \|\hat{u}^{n+1}\|_{\infty}\leq 1.
\end{align}
Thus,
\begin{align}
 \|U^{n+1}\|_{\infty}= \|\hat{u}^{n+1}\|_{\infty}\leq 1.
\end{align}
This completes the proof of the theorem.
\end{proof}

\begin{remark}
The first entry in the first row and the last entry in the last row in both matrices $\mathbf{C}_{x(y,z)}$ and $\mathbf{D}_{x(y,z)}$ can take arbitrary numbers due to homogeneous Dirichlet boundary conditions. Here we set these two elements the same as the diagonal entry of matrices $\mathbf{C}_{x(y,z)}$ and $\mathbf{D}_{x(y,z)}$ so that Lemma~\ref{Lemma3} can be directly applied when obtaining the estimate (\ref{condition4}).
\end{remark}

The theoretical results in previous sections also hold for the second-order scheme (\ref{stage1}), (\ref{ADI11})-(\ref{ADI31}) and (\ref{stage3}) with some minor changes. For example, Theorem \ref{th4} will become
\begin{theorem}
Assume that the initial value $u_0(x)$ satisfies $\max\limits_{x\in\bar{\Omega}}|u_0(x)|\leq 1$, then the numerical solution $U^{n}_{i,j,k}$ of
(\ref{stage1}), (\ref{ADI11})-(\ref{ADI31}) and (\ref{stage3}) satisfies the discrete maximum principle, i.e.,  $\|U^n\|_{\infty}\leq 1$ for any $n=0,1,\cdots,N$
if the time step satisfies
 \begin{align}
 \Delta t\leq 2\frac{\min(h^{\alpha}_1,h^{\alpha}_2,h^{\alpha}_3)}{\varepsilon^2c^{\alpha}_0}.
  \end{align}
\end{theorem}
\begin{proof}
  To prove the above theorem, the main difference is that the matrices $\mathbf{D}_x$, $\mathbf{D}_y$ and $\mathbf{D}_z$ should be changed into a zero matrix. Other parts of the proof are basically the same as Theorem \ref{th4}.
\end{proof}

\section{Numerical results}\label{sec5}
Our code is written in Matlab and the programs are carried out  on a desktop with Intel CPU i7-4790K (4.00GHz) and  16GB RAM.

\subsection{Convergence and stability study}
In order to numerically test the accuracy of the numerical method, we use exact solutions with sufficient regularity in this subsection as the testing examples.

\begin{example}\label{ex1}
In this example, we consider the 2D  space fractional Allen-Cahn equation with the exact solution
\begin{align}
u(x,y,t)=e^{-t}x^4(1-x)^4y^4(1-y)^4,
\end{align}
so that the exact solution has a sufficient regularity. And in this example, the equation needs to be modified with a source term
\begin{align*}
f(x,y,t)&=\frac{\varepsilon^2}{2\cos(\frac{\alpha\pi}{2})} e^{-t}\Big[\frac{\Gamma(5)}{\Gamma(5-\alpha)}(x^{4-\alpha}+(1-x)^{4-\alpha})- \frac{4\Gamma(6)}{\Gamma(6-\alpha)}(x^{5-\alpha}+(1-x)^{5-\alpha})\\
&+\frac{6\Gamma(7)}{\Gamma(7-\alpha)}(x^{6-\alpha}+(1-x)^{6-\alpha})-\frac{4\Gamma(8)}{\Gamma(8-\alpha)}(x^{7-\alpha}+(1-x)^{7-\alpha})+\frac{\Gamma(9)}{\Gamma(9-\alpha)}(x^{8-\alpha}+(1-x)^{8-\alpha})\Big]y^4(1-y)^4\\
&+\frac{\varepsilon^2}{2\cos(\frac{\alpha\pi}{2})}e^{-t}\Big[\frac{\Gamma(5)}{\Gamma(5-\alpha)}(y^{4-\alpha}+(1-y)^{4-\alpha})-\frac{4\Gamma(6)}{\Gamma(6-\alpha)}(y^{5-\alpha}+(1-y)^{5-\alpha})\\
&+\frac{6\Gamma(7)}{\Gamma(7-\alpha)}(y^{6-\alpha}+(1-y)^{6-\alpha})-\frac{4\Gamma(8)}{\Gamma(8-\alpha)}(y^{7-\alpha}+(1-y)^{7-\alpha})+\frac{\Gamma(9)}{\Gamma(9-\alpha)}(y^{8-\alpha}+(1-y)^{8-\alpha})\Big]x^4(1-x)^4 \\& + e^{-3t}x^{12}(1-x)^{12}y^{12}(1-y)^{12}-2e^{-t}x^4(1-x)^4y^4(1-y)^4.
\end{align*}
The initial condition is given according to this exact solution and $\varepsilon$ is set to be 0.1.
\end{example}

We carry out numerical accuracy test for $1<\alpha\leq2$. We measure the numerical errors $e_1(\Delta t,h_x,h_y,h_z)=u-U(\Delta t, h_x,h_y,h_z)$
and  $e_2(\Delta t,h_x,h_y,h_z)=u-\widetilde{U}(\Delta t, h_x,h_y,h_z)$ at time $T=1$ in the $L^{\infty}$-norm, and compute the convergence orders according to
\begin{align*}
 \textrm{order}_1=\log_2\left(\frac{\parallel e_1(\Delta t,h_x,h_y,h_z)\parallel_{\infty}}{\parallel e_1({\Delta t}/{2},{h_x}/{2},{h_y}/{2},{h_z}/{2})\parallel_{\infty}}\right),
\end{align*}
and
\begin{align*}
 \textrm{order}_2=\log_2\left(\frac{\parallel e_2(\Delta t,h_x,h_y,h_z)\parallel_{\infty}}{\parallel e_2({\Delta t}/{2},{h_x}/{2},{h_y}/{2},{h_z}/{2})\parallel_{\infty}}\right).
\end{align*}

\begin{table}[!tbp]
\tabcolsep=8pt
\caption{$L^\infty$-norm errors and CPU times (in seconds) for Example \ref{ex1} with $\alpha=1.2$.}\label{table11}
 \centering
  \begin{tabular}{lllllll}
    \hline
 $\Delta t$ & $h$ &  CPU & $\parallel U^N-u^N\parallel_\infty$ &  order$_1$ & $\parallel \widetilde{U}^N-u^N\parallel_\infty$  & order$_2$    \\
\hline
$1/16$  &$1/16$  &  0.02s  &  1.79E$-$08 &        & 5.45E$-$10&             \\
$1/32$  &$1/32$  &  0.07s  &  4.37E$-$09 & 2.03   & 3.41E$-$11& 4.00        \\
$1/64$  &$1/64$  &  0.23s  &  1.09E$-$09 & 2.01   & 4.19E$-$12& 3.02        \\
$1/128$ &$1/128$ &  0.98s  &  2.71E$-$10 & 2.00   & 3.27E$-$13& 3.68        \\
$1/256$ &$1/256$ &  7.44s  &  6.78E$-$11 & 2.00   & 2.50E$-$14& 3.71        \\
\hline
\end{tabular}
\end{table}

\begin{table}[!tbp]
\tabcolsep=8pt
\caption{$L^\infty$-norm errors and CPU times (in seconds) for Example \ref{ex1} with $\alpha=1.5$.}\label{table12}
 \centering
  \begin{tabular}{lllllll}
    \hline
 $\Delta t$ & $h$ &  CPU & $\parallel U^N-u^N\parallel_\infty$ &  order$_1$ & $\parallel \widetilde{U}^N-u^N\parallel_\infty$  & order$_2$    \\
\hline
$1/16$  &$1/16$  &  0.02s   &  1.48E$-$08 &        & 1.04E$-$09&             \\
$1/32$  &$1/32$  &  0.07s   &  3.51E$-$09 & 2.08   & 6.48E$-$11& 4.00        \\
$1/64$  &$1/64$  &  0.23s   &  8.66E$-$10 & 2.02   & 4.06E$-$12& 4.00        \\
$1/128$ &$1/128$ &  0.97s   &  2.16E$-$10 & 2.00   & 2.54E$-$13& 4.00        \\
$1/256$ &$1/256$ &  7.43s   &  5.39E$-$11 & 2.00   & 1.72E$-$14& 3.88        \\
\hline
\end{tabular}
\end{table}

\begin{table}[!tbp]
\tabcolsep=8pt
\caption{$L^\infty$-norm errors and CPU times (in seconds) for Example \ref{ex1} with $\alpha=1.8$.}\label{table13}
 \centering
  \begin{tabular}{lllllll}
    \hline
 $\Delta t$ & $h$ &  CPU & $\parallel U^N-u^N\parallel_\infty$ &  order$_1$ & $\parallel \widetilde{U}^N-u^N\parallel_\infty$  & order$_2$    \\
\hline
$1/16$  &$1/16$  &  0.02s  &  1.08E$-$08 &        & 1.96E$-$09&             \\
$1/32$  &$1/32$  &  0.06s  &  2.37E$-$09 & 2.19   & 1.17E$-$10& 4.06        \\
$1/64$  &$1/64$  &  0.23s  &  5.71E$-$10 & 2.05   & 7.03E$-$12& 4.06        \\
$1/128$ &$1/128$ &  0.97s  &  1.41E$-$10 & 2.01   & 4.30E$-$13& 4.03        \\
$1/256$ &$1/256$ &  7.43s  &  3.53E$-$11 & 2.00   & 2.69E$-$14& 4.00        \\
\hline
\end{tabular}
\end{table}

\begin{table}[!tbp]
\tabcolsep=8pt
\caption{$L^\infty$-norm errors and CPU times (in seconds) for Example \ref{ex1} with $\alpha=2.0$.}\label{table14}
 \centering
  \begin{tabular}{lllllll}
    \hline
 $\Delta t$ & $h$ &  CPU & $\parallel U^N-u^N\parallel_\infty$ &  order$_1$ & $\parallel \widetilde{U}^N-u^N\parallel_\infty$  & order$_2$    \\
\hline
$1/16$  &$1/16$  &   0.02s  &  7.94E$-$09 &        & 2.88E$-$09&             \\
$1/32$  &$1/32$  &   0.06s  &  1.56E$-$09 & 2.34   & 1.80E$-$10& 3.99        \\
$1/64$  &$1/64$  &   0.23s  &  3.78E$-$10 & 2.05   & 1.13E$-$11& 4.00        \\
$1/128$ &$1/128$ &   0.97s  &  9.39E$-$11 & 2.01   & 7.06E$-$13& 4.00        \\
$1/256$ &$1/256$ &   7.43s  &  2.35E$-$11 & 2.00   & 4.42E$-$14& 4.00        \\
\hline
\end{tabular}
\end{table}
\begin{table}[!tbp]
\tabcolsep=6pt
\caption{$L^\infty$-norm errors and CPU times (in seconds) for Example \ref{ex1} with  $\alpha=1.5$ using unequal meshsizes in $x$ and $y$ directions.}\label{table15}
 \centering
  \begin{tabular}{llllllll}
    \hline
 $\Delta t$ & $h_x$ & $h_y$ &  CPU & $\parallel U^N-u^N\parallel_\infty$ &  order$_1$ & $\parallel \widetilde{U}^N-u^N\parallel_\infty$  & order$_2$    \\
\hline
$1/16$  &$1/16$  &$1/32$ &  0.03s   &  1.43E$-$08 &        & 9.83E$-$09&             \\
$1/32$  &$1/32$  &$1/64$ &  0.10s   &  3.48E$-$09 & 2.04   & 5.89E$-$11& 4.06        \\
$1/64$  &$1/64$  &$1/128$&  0.29s   &  8.64E$-$10 & 2.01   & 3.63E$-$12& 4.02        \\
$1/128$ &$1/128$ &$1/256$&  1.94s   &  2.16E$-$10 & 2.00   & 2.22E$-$13& 4.03        \\
$1/256$ &$1/256$ &$1/512$&  17.0s   &  5.39E$-$11 & 2.00   & 1.72E$-$14& 3.69        \\
\hline
\end{tabular}
\end{table}

Table~\ref{table11}-Table~\ref{table14} list the errors and the corresponding convergence orders for $\alpha= 1.2, 1.5, 1.8, 2$ in the $L^\infty$-norm using the same spatial meshsize $h=h_x=h_y=h_z$ while Table~\ref{table15} lists the errors and the corresponding convergence orders for $\alpha=1.5$ in the $L^\infty$-norm using different spatial meshsizes. As we can see that these results confirm second-order accuracy in time variable and fourth-order accuracy in space variable if the Richardson extrapolation (\ref{Richardson}) is not applied. But the results is fourth-order accurate both in time and  space variables if the Richardson extrapolation (\ref{Richardson}) is applied. Additionally, the computational time in seconds is also provided in Table~\ref{table11}-Table~\ref{table15}, as we can see that the computational time for $\Delta t=h=\frac{1}{256}$ is less than 10 seconds, and the computational time for $\Delta t=h_x=\frac{1}{256},h_y=\frac{1}{512}$ is less than 20 seconds. And the method is extremely accurate, the error between the extrapolated solution and exact solution is in the order of $10^{-14}$ when $\Delta t=h=\frac{1}{256}$ and $\Delta t=h_x=\frac{1}{256},h_y=\frac{1}{512}$, which is nearly the machine accuracy.

To show the unconditional stability, we fix $h$ and vary $\Delta t$, results for $\alpha=1.2$ and $\alpha=1.8$ are plotted in Figure~\ref{Fig1a}. As one can see that these results clearly show that the time step is not related to the spatial meshsize, and as the spatial meshsize goes to zero, the dominant error comes from the temporal part.

\begin{figure}[!tbp]
\centering
\subfigure[$\alpha=1.2$]{
\includegraphics[width=0.45\textwidth]{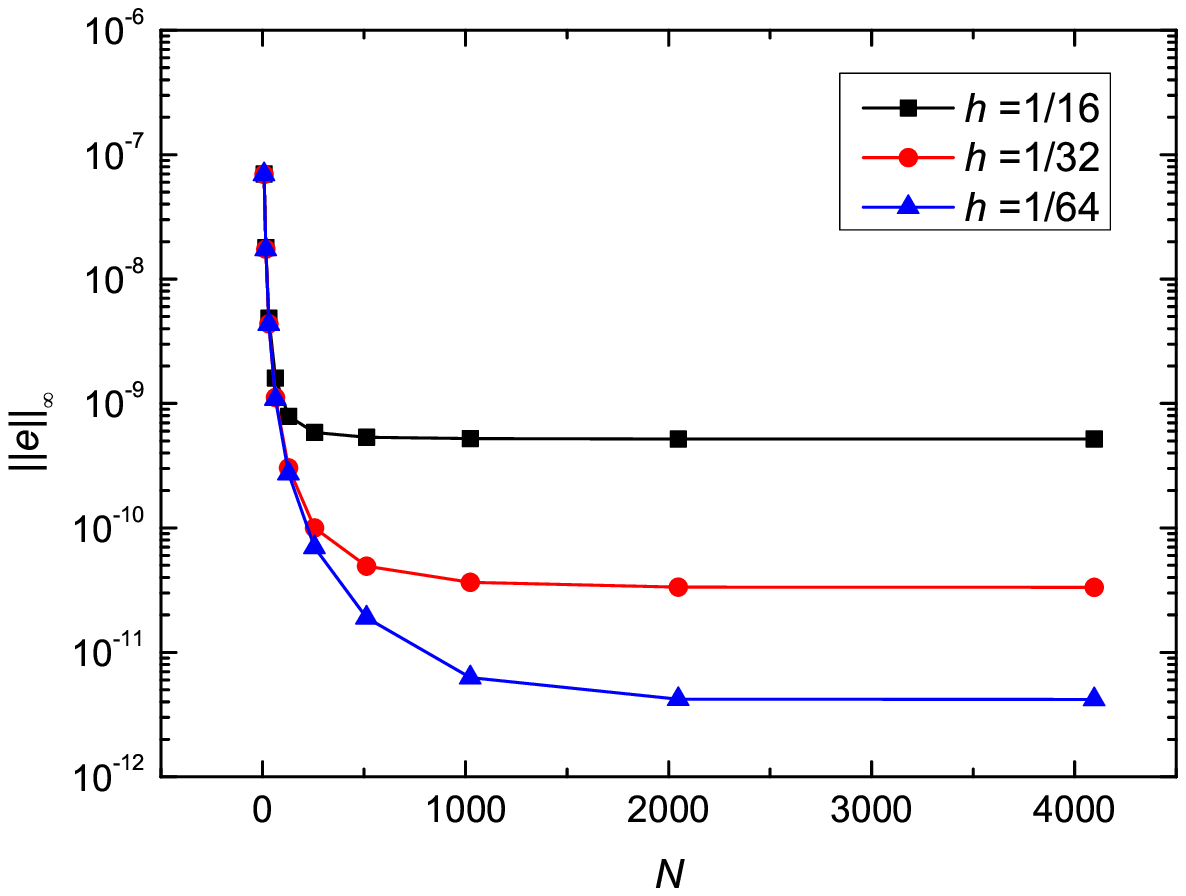}}
\subfigure[$\alpha=1.8$]{
\includegraphics[width=0.45\textwidth]{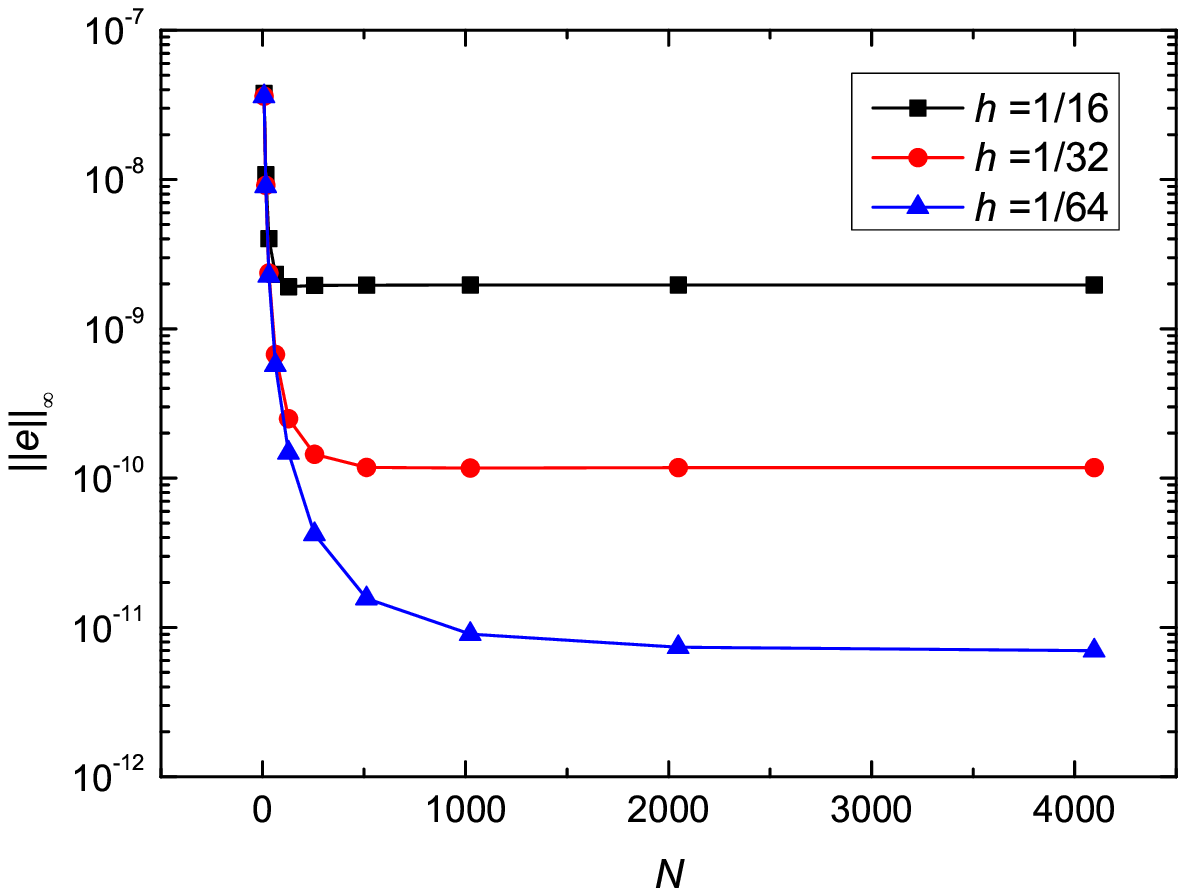}} \\
\caption{Numerical results  for Example \ref{ex1} with fixed $h$ but varying $\Delta t$.}
\label{Fig1a}
\end{figure}


\begin{example}\label{ex2}
In this example, we consider the 3D  space fractional Allen-Cahn equation with the exact solution
\begin{align}
u(x,y,t)=e^{-t}x^4(1-x)^4y^4(1-y)^4z^4(1-z)^4.
\end{align}
The source term and initial condition are given according to this exact solution,  in addition,  $\varepsilon$ is set to be 0.1.
\end{example}

Again, we carry out numerical accuracy test for $1<\alpha\leq2$. Table~\ref{table21}-Table~\ref{table24} list the errors and the corresponding convergence orders for $\alpha = 1.2, 1.5, 1.8, 2$ in the $L^\infty$-norm using the same spatial meshsize $h=h_x=h_y=h_z$ while Table~\ref{table15} lists the errors and the corresponding convergence orders for $\alpha=1.5$ in the $L^\infty$-norm using different spatial meshsizes. As we can see that these results confirm second-order accuracy in time variable and fourth-order accuracy in space variable if the Richardson extrapolation (\ref{Richardson}) is not applied. But the results is fourth-order accurate both in time and  space variables if the Richardson extrapolation (\ref{Richardson}) is applied. Additionally, the computational time in seconds is also provided in Table~\ref{table21}-Table~\ref{table25}, as we can see that the computational time for $\Delta t=h=\frac{1}{128}$ is less than 120 seconds and the computational time for $\Delta t=h_x=\frac{1}{128},h_y=\frac{1}{160}, h_x=\frac{1}{256}$ is less than $300$ seconds. And the method is extremely accurate, the error between the extrapolated solution and exact solution is in the order of $10^{-15}$ when $\Delta t=h=\frac{1}{128}$ and in the order of $10^{-16}$ when $\Delta t=h_x=\frac{1}{128},h_y=\frac{1}{160}, h_x=\frac{1}{256}$, which are both nearly the machine accuracy.

To show the unconditional stability, we fix $h$ and vary $\Delta t$, results for $\alpha=1.2$ and $\alpha=1.8$ are plotted in Figure~\ref{Fig2a}. As one can see that these results clearly show that the time step is not related to the spatial meshsize, and as the spatial meshsize goes to zero, the dominant error comes from the temporal part.

\begin{table}[!tbp]
\tabcolsep=8pt
\caption{$L^\infty$-norm errors and CPU times (in seconds) for Example \ref{ex2} with $\alpha=1.2$.}\label{table21}
 \centering
  \begin{tabular}{lllllll}
    \hline
 $\Delta t$ & $h$ &  CPU & $\parallel U^N-u^N\parallel_\infty$ &  order$_1$ & $\parallel \widetilde{U}^N-u^N\parallel_\infty$  & order$_2$    \\
\hline
$1/8$   &$1/8$   &  0.01s  &  2.79E$-$10 &        & 4.82E$-$11&             \\
$1/16$  &$1/16$  &  0.07s  &  6.10E$-$11 & 2.19   & 2.99E$-$12& 4.01        \\
$1/32$  &$1/32$  &  0.42s  &  1.47E$-$11 & 2.05   & 1.87E$-$13& 4.00        \\
$1/64$  &$1/64$  &  6.28s  &  3.65E$-$12 & 2.01   & 1.58E$-$14& 3.56        \\
$1/128$ &$1/128$ &  110s   &  9.09E$-$13 & 2.00   & 1.24E$-$15& 3.68        \\
\hline
\end{tabular}
\end{table}

\begin{table}[!tbp]
\tabcolsep=8pt
\caption{$L^\infty$-norm errors and CPU times (in seconds) for Example \ref{ex2} with $\alpha=1.5$.}\label{table22}
 \centering
  \begin{tabular}{lllllll}
    \hline
 $\Delta t$ & $h$ &  CPU & $\parallel U^N-u^N\parallel_\infty$ &  order$_1$ & $\parallel \widetilde{U}^N-u^N\parallel_\infty$  & order$_2$    \\
\hline
$1/8$   &$1/8$   &  0.01s  &  2.53E$-$10 &        & 9.13E$-$11&             \\
$1/16$  &$1/16$  &  0.07s  &  4.61E$-$11 & 2.46   & 5.61E$-$12& 4.02        \\
$1/32$  &$1/32$  &  0.42s  &  1.05E$-$11 & 2.14   & 3.50E$-$13& 4.00        \\
$1/64$  &$1/64$  &  6.27s  &  2.56E$-$12 & 2.04   & 2.19E$-$14& 4.00        \\
$1/128$ &$1/128$ &  110s   &  6.35E$-$13 & 2.01   & 1.37E$-$15& 4.00        \\
\hline
\end{tabular}
\end{table}

\begin{table}[!tbp]
\tabcolsep=8pt
\caption{$L^\infty$-norm errors and CPU times (in seconds) for Example \ref{ex2} with $\alpha=1.8$.}\label{table23}
 \centering
  \begin{tabular}{lllllll}
    \hline
 $\Delta t$ & $h$ &  CPU & $\parallel U^N-u^N\parallel_\infty$ &  order$_1$ & $\parallel \widetilde{U}^N-u^N\parallel_\infty$  & order$_2$    \\
\hline
$1/8$   &$1/8$   &  0.01s  &  2.31E$-$10 &        & 1.51E$-$10&             \\
$1/16$  &$1/16$  &  0.07s  &  2.90E$-$11 & 2.99   & 9.19E$-$12& 4.04        \\
$1/32$  &$1/32$  &  0.42s  &  5.53E$-$12 & 2.39   & 5.71E$-$13& 4.01        \\
$1/64$  &$1/64$  &  6.26s  &  1.28E$-$12 & 2.12   & 3.57E$-$14& 4.00        \\
$1/128$ &$1/128$ &  110s   &  3.12E$-$13 & 2.03   & 2.23E$-$15& 4.00        \\
\hline
\end{tabular}
\end{table}

\begin{table}[!tbp]
\tabcolsep=8pt
\caption{$L^\infty$-norm errors and CPU times (in seconds) for Example \ref{ex2} with $\alpha=2.0$.}\label{table24}
 \centering
  \begin{tabular}{lllllll}
    \hline
 $\Delta t$ & $h$ &  CPU & $\parallel U^N-u^N\parallel_\infty$ &  order$_1$ & $\parallel \widetilde{U}^N-u^N\parallel_\infty$  & order$_2$    \\
\hline
$1/8$   &$1/8$   &  0.01s  &  2.22E$-$10 &        & 1.91E$-$10 &            \\
$1/16$  &$1/16$  &  0.06s  &  1.92E$-$11 & 3.53   & 1.15E$-$11 & 4.05       \\
$1/32$  &$1/32$  &  0.42s  &  2.93E$-$12 & 2.71   & 7.13E$-$13 & 4.01       \\
$1/64$  &$1/64$  &  6.23s  &  7.17E$-$13 & 2.03   & 4.44E$-$14 & 4.00       \\
$1/128$ &$1/128$ &  109s  &  1.79E$-$13 & 2.00   & 2.78E$-$15 & 4.00       \\
\hline
\end{tabular}
\end{table}

\begin{table}[!tbp]
\tabcolsep=5pt
\caption{$L^\infty$-norm errors and CPU times (in seconds) for Example \ref{ex2} with  $\alpha=1.5$ using unequal meshsizes in $x$, $y$ and $z$ directions.}\label{table25}
 \centering
  \begin{tabular}{lllllllll}
    \hline
 $\Delta t$ & $h_x$ & $h_y$ & $h_z$ &  CPU & $\parallel U^N-u^N\parallel_\infty$ &  order$_1$ & $\parallel \widetilde{U}^N-u^N\parallel_\infty$  & order$_2$    \\
\hline
$1/8$   &$1/8$   &$1/10$ &$1/16$ &  0.02s   &  2.06E$-$10 &        & 5.94E$-$11&             \\
$1/16$  &$1/16$  &$1/20$ &$1/32$ &  0.10s   &  4.33E$-$11 & 2.25   & 3.53E$-$12& 4.07        \\
$1/32$  &$1/32$  &$1/40$ &$1/64$ &  0.96s   &  1.03E$-$11 & 2.07   & 2.11E$-$13& 4.06        \\
$1/64$  &$1/64$  &$1/80$ &$1/128$&  17.2s   &  2.54E$-$12 & 2.02   & 1.30E$-$14& 4.03        \\
$1/128$ &$1/128$ &$1/160$&$1/256$&  280.s   &  6.34E$-$13 & 2.00   & 8.29E$-$16& 3.97        \\
\hline
\end{tabular}
\end{table}

\begin{figure}[!tbp]
\centering
\subfigure[$\alpha=1.2$]{
\includegraphics[width=0.45\textwidth]{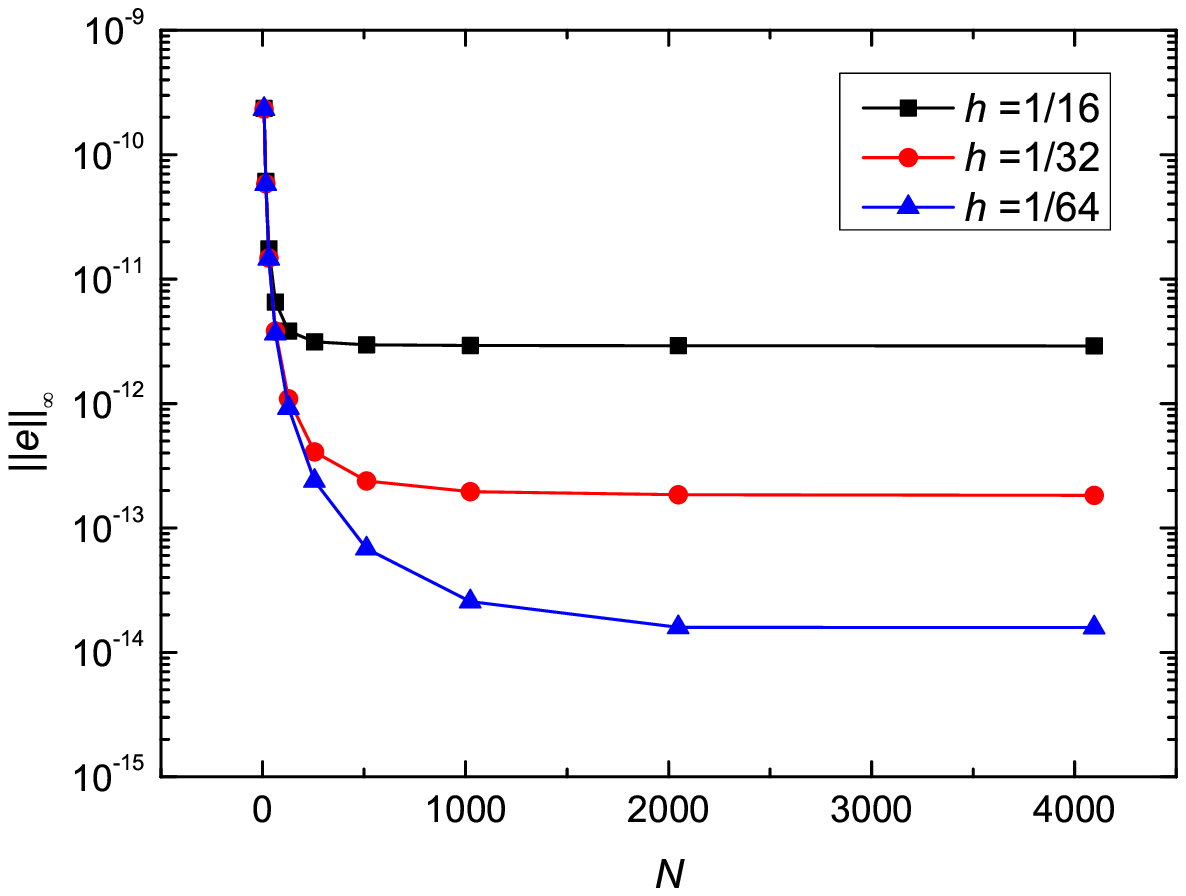}}
\subfigure[$\alpha=1.8$]{
\includegraphics[width=0.45\textwidth]{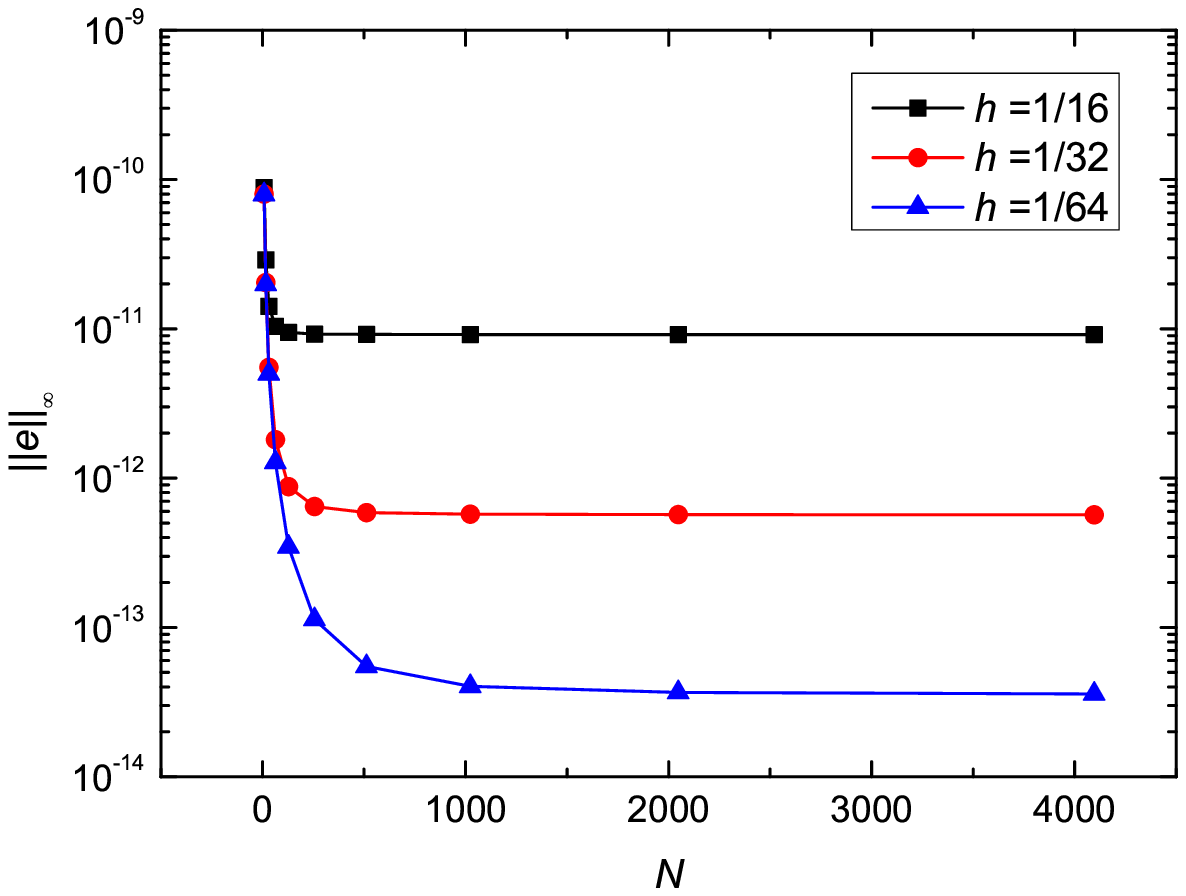}} \\
\caption{Numerical results  for Example \ref{ex2} with fixed $h$ but varying $\Delta t$.}
\label{Fig2a}
\end{figure}

\subsection{Numerical tests for discrete maximum principle}

\begin{example}\label{ex3}
In this example, we consider the 2D  space fractional Allen-Cahn equation with initial condition
\begin{align}
u_0(x,y)=0.95\times rand(x,y)+0.05,
\end{align}
where zero boundary values are set for the initial condition $u_0(x,y)$. Moreover,  $\alpha$ is set to be $1.7$.
\end{example}

For this example, we fix $h=0.05$ bur vary $\varepsilon$ and $\Delta t$. For $\varepsilon=0.1$, the  maximum principle condition (\ref{condition}) requires $0.1508\leq \Delta t\leq 0.4358$. The top four sub-figures in Figure~\ref{Fig2} show that the maximum values of the numerical solutions are bounded by $1$ if $\Delta t=0.01$, $\Delta t=0.4$, and $\Delta t=3$ but exceed 1  if $\Delta t=4$.  For $\varepsilon=0.2$, the  maximum principle condition (\ref{condition}) requires $0.0377\leq \Delta t\leq 0.1089$. The lower four sub-figures in Figure~\ref{Fig2} show the maximum values of the numerical solutions are bounded by $1$ if $\Delta t=0.001$, $\Delta t=0.1$, and $\Delta t=0.8$ but exceed 1  if $\Delta t=1.2$. These numerical results suggest that the constraint (\ref{condition}) for time step size to achieve the discrete maximum principle is only a sufficient condition. In practice, the maximum principle is still valid if a time step size with much smaller values or larger values is adopted.

\begin{figure}[!tbp]
\centering
{
\includegraphics[width=0.23\textwidth]{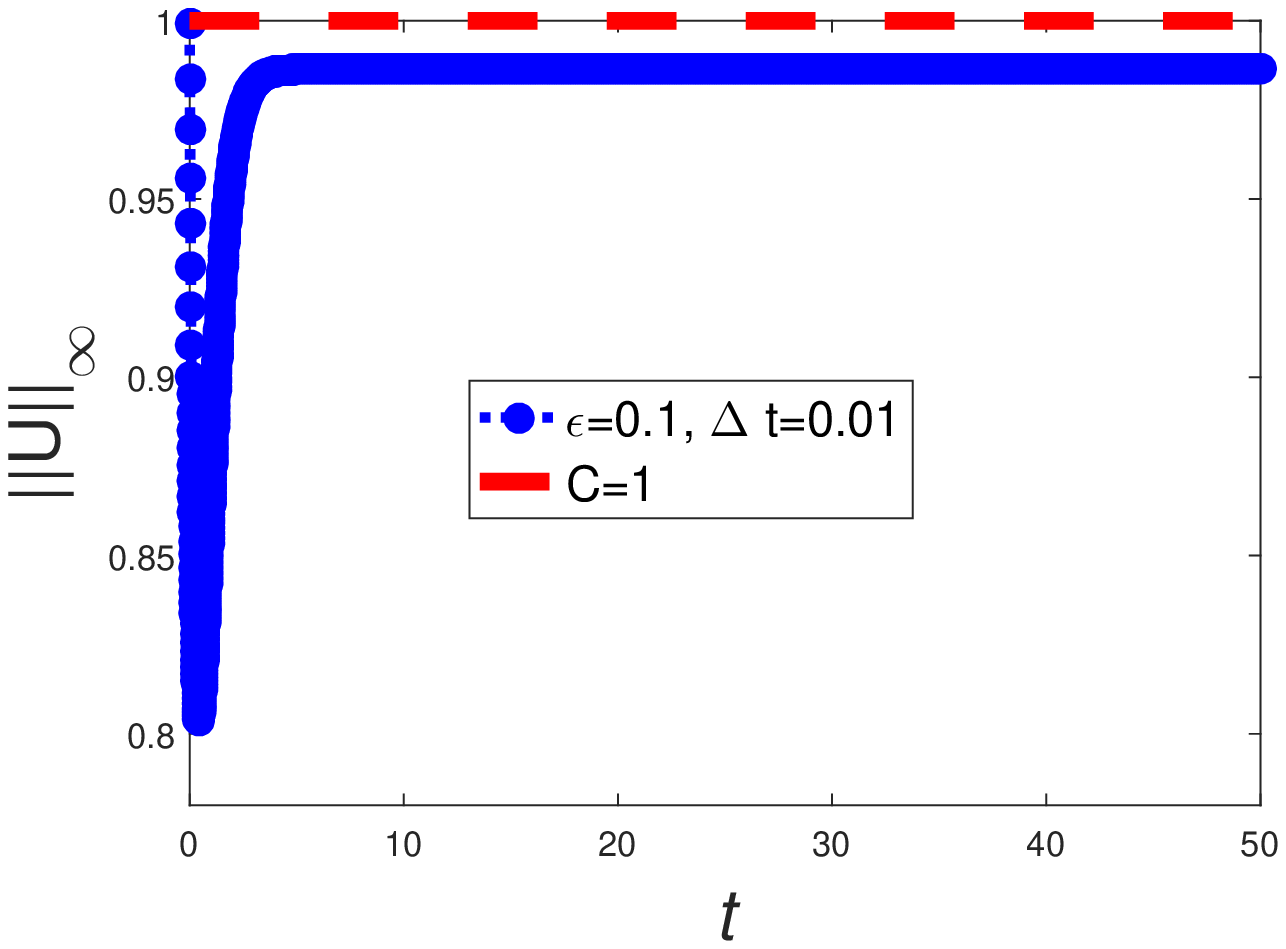}}
{
\includegraphics[width=0.23\textwidth]{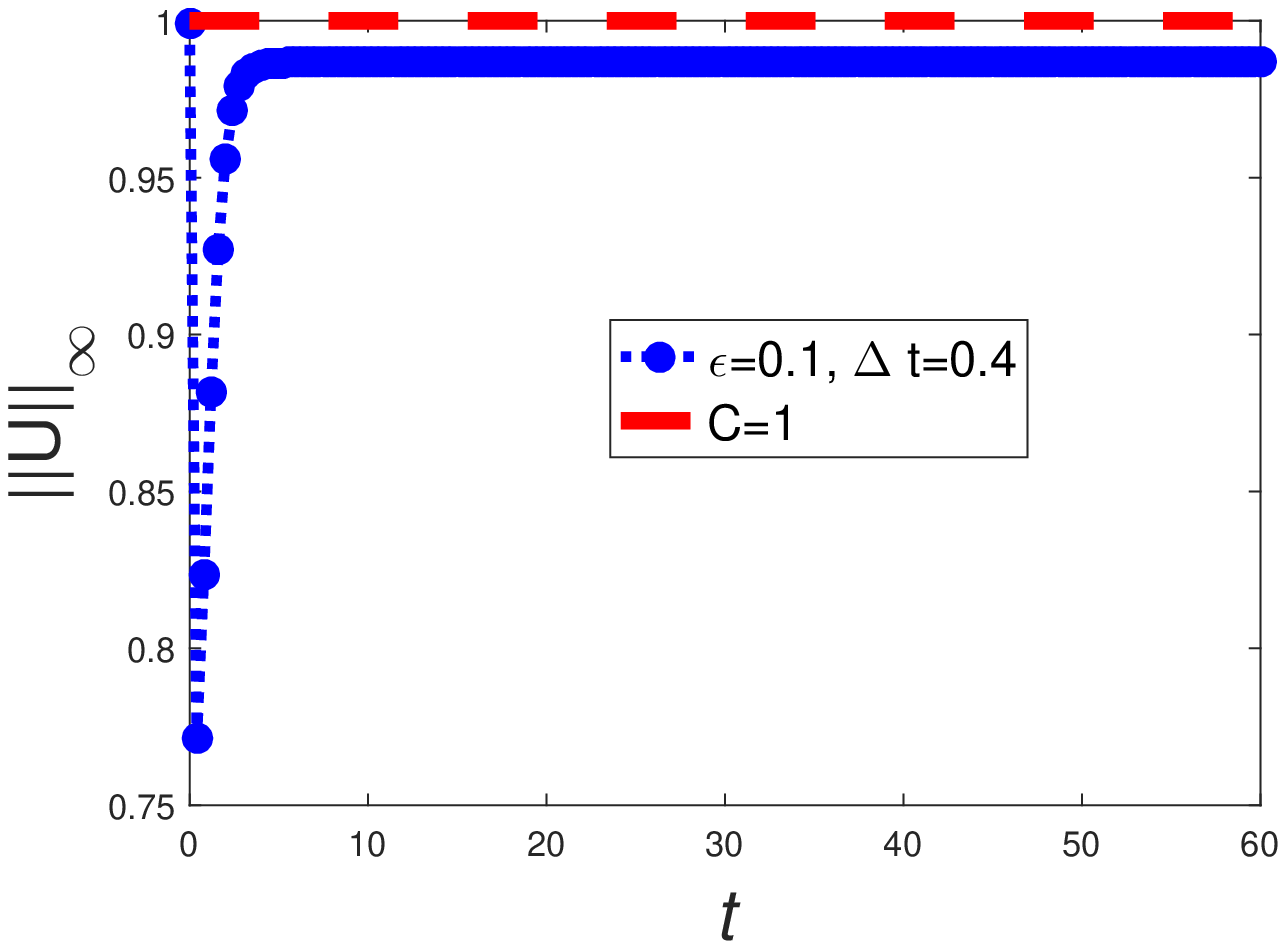}}
{
\includegraphics[width=0.23\textwidth]{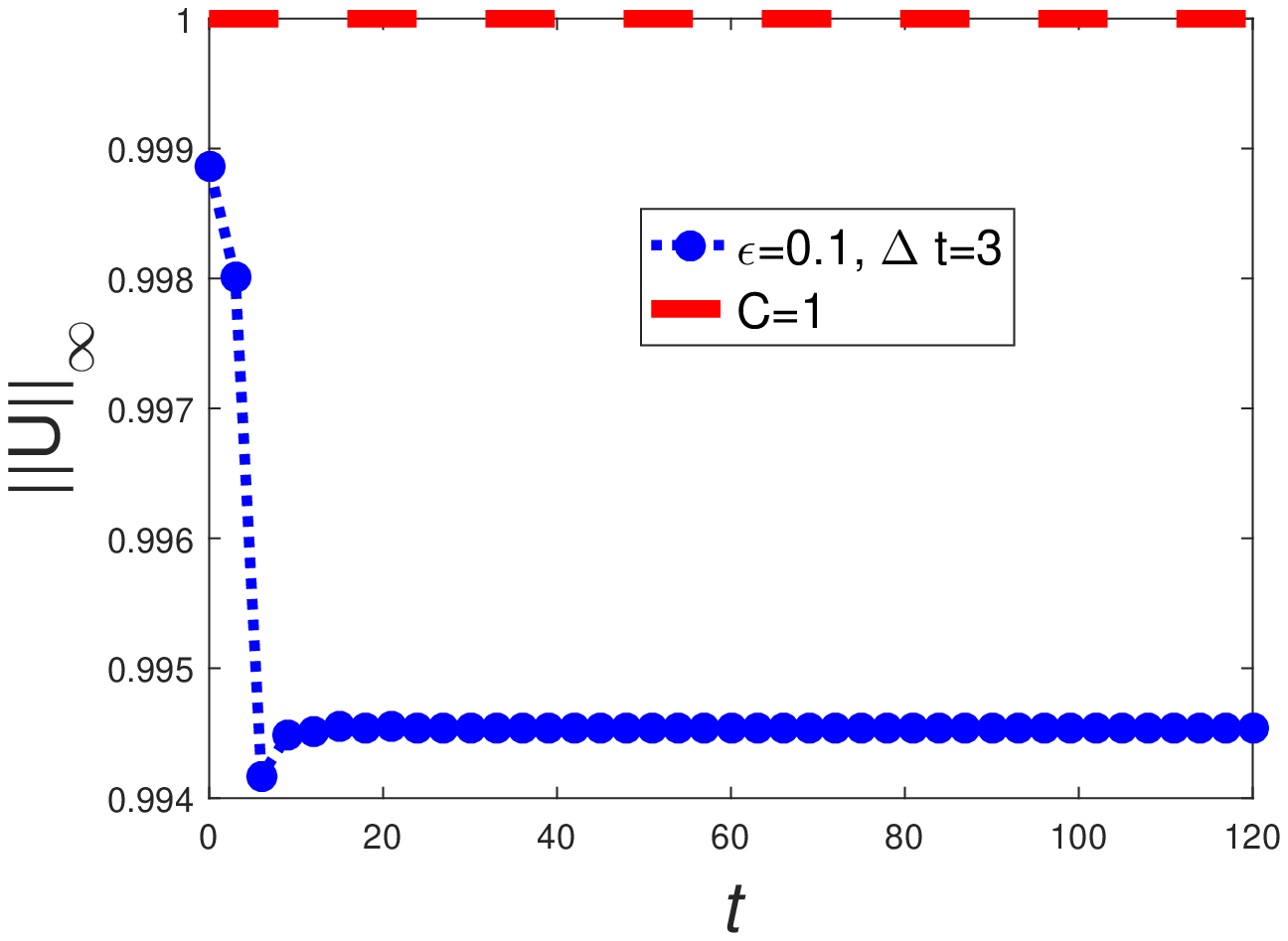}}
{
\includegraphics[width=0.23\textwidth]{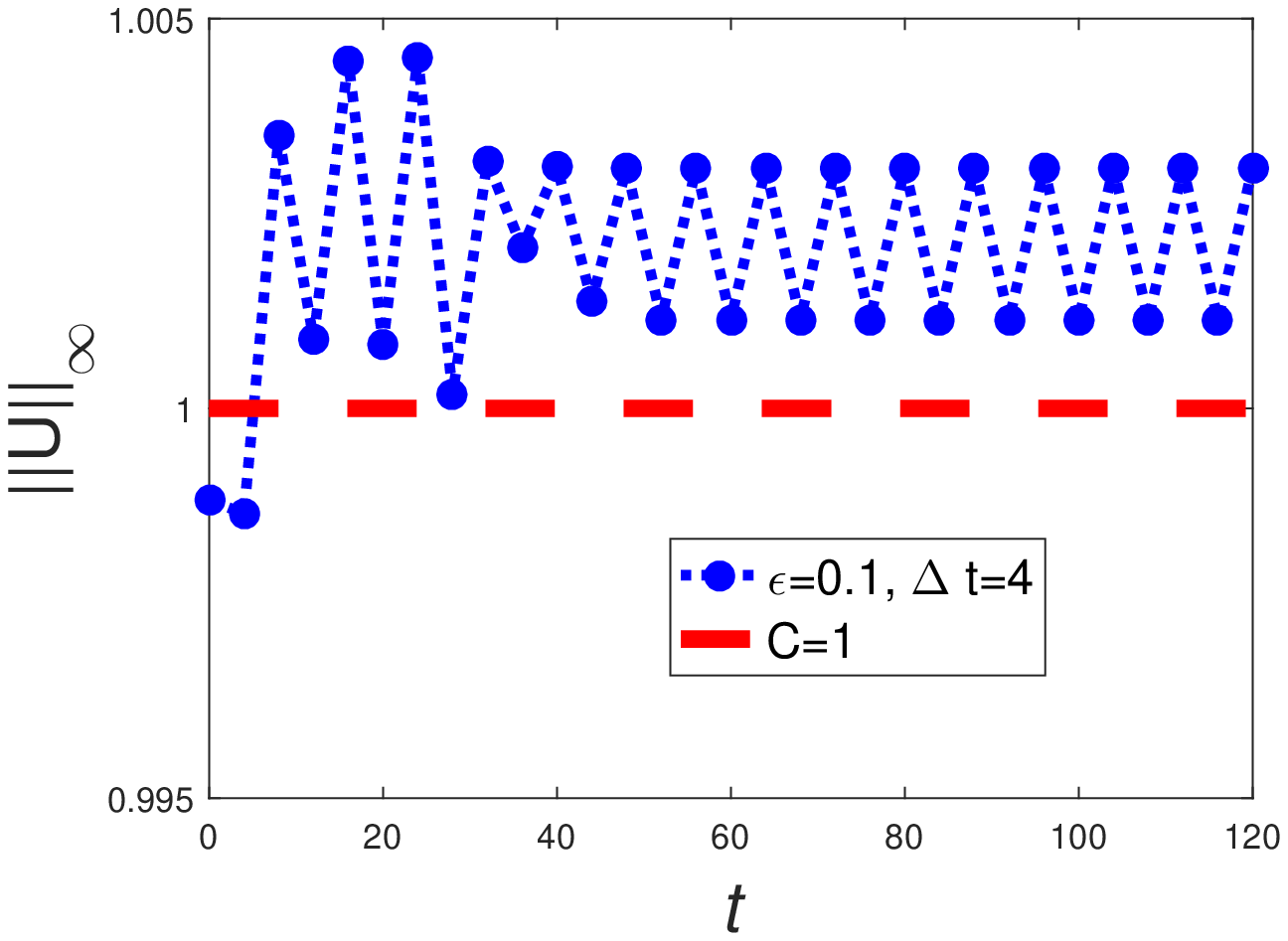}}\\
{
\includegraphics[width=0.23\textwidth]{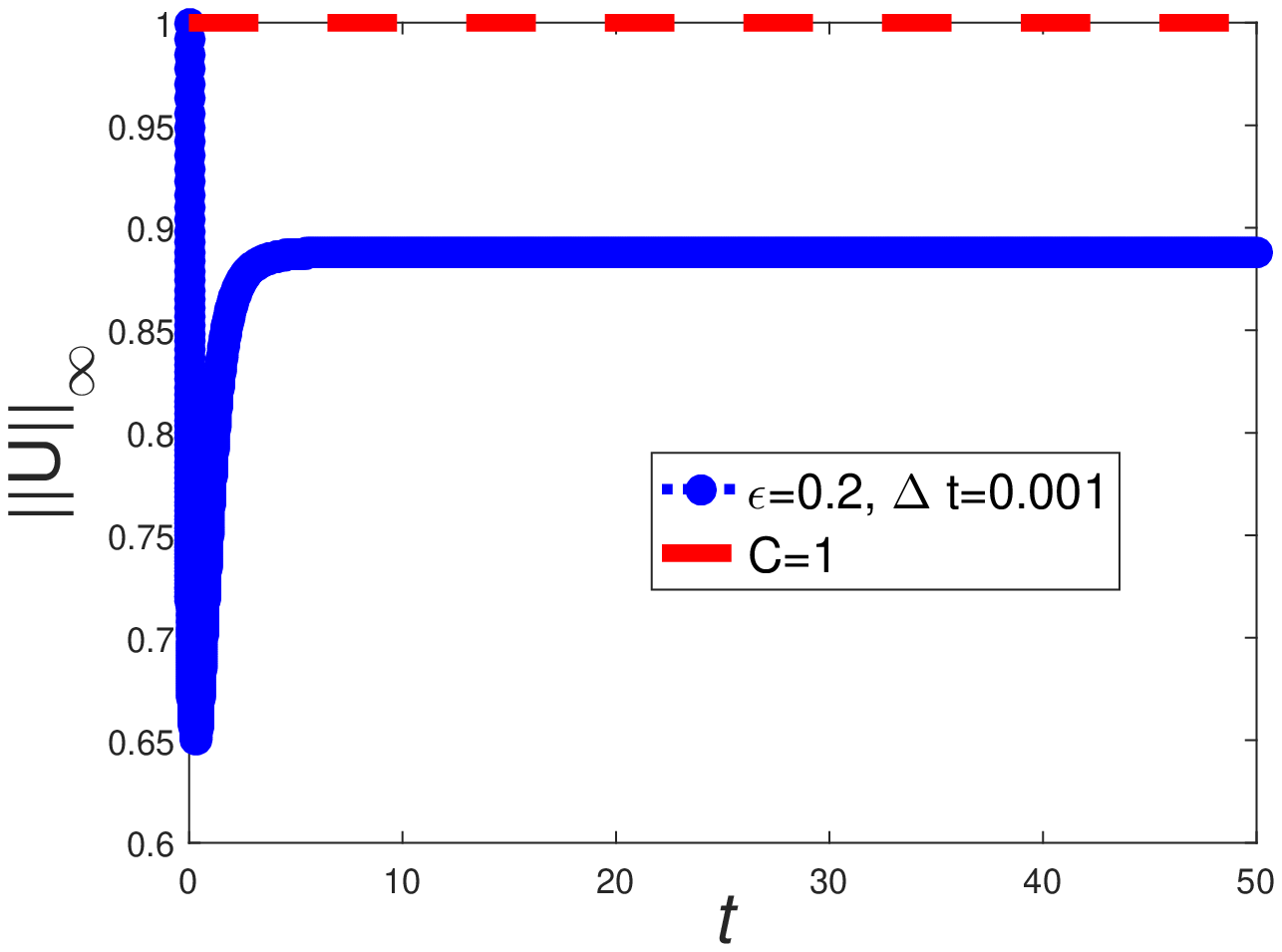}}
{
\includegraphics[width=0.23\textwidth]{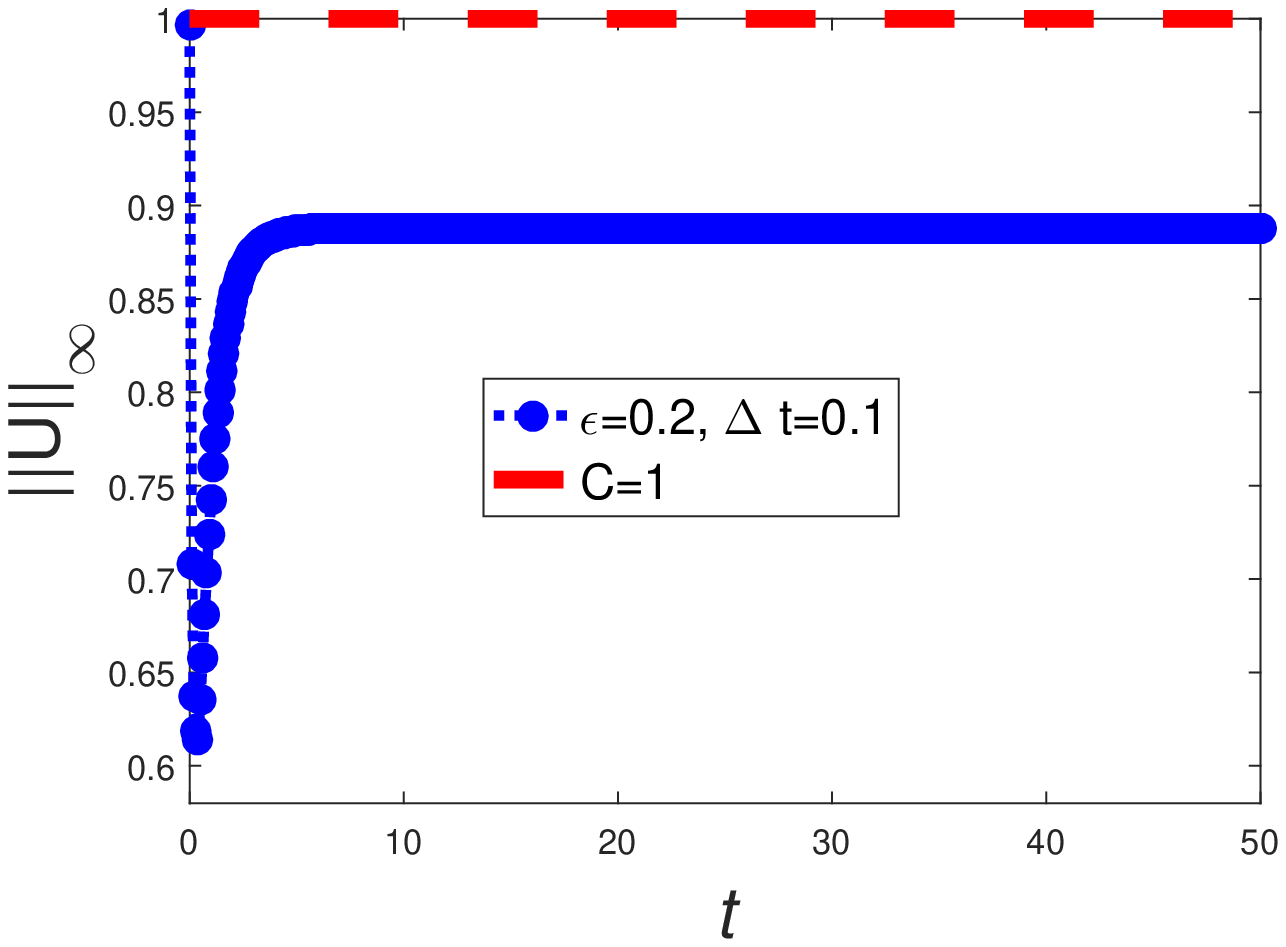}}
{
\includegraphics[width=0.23\textwidth]{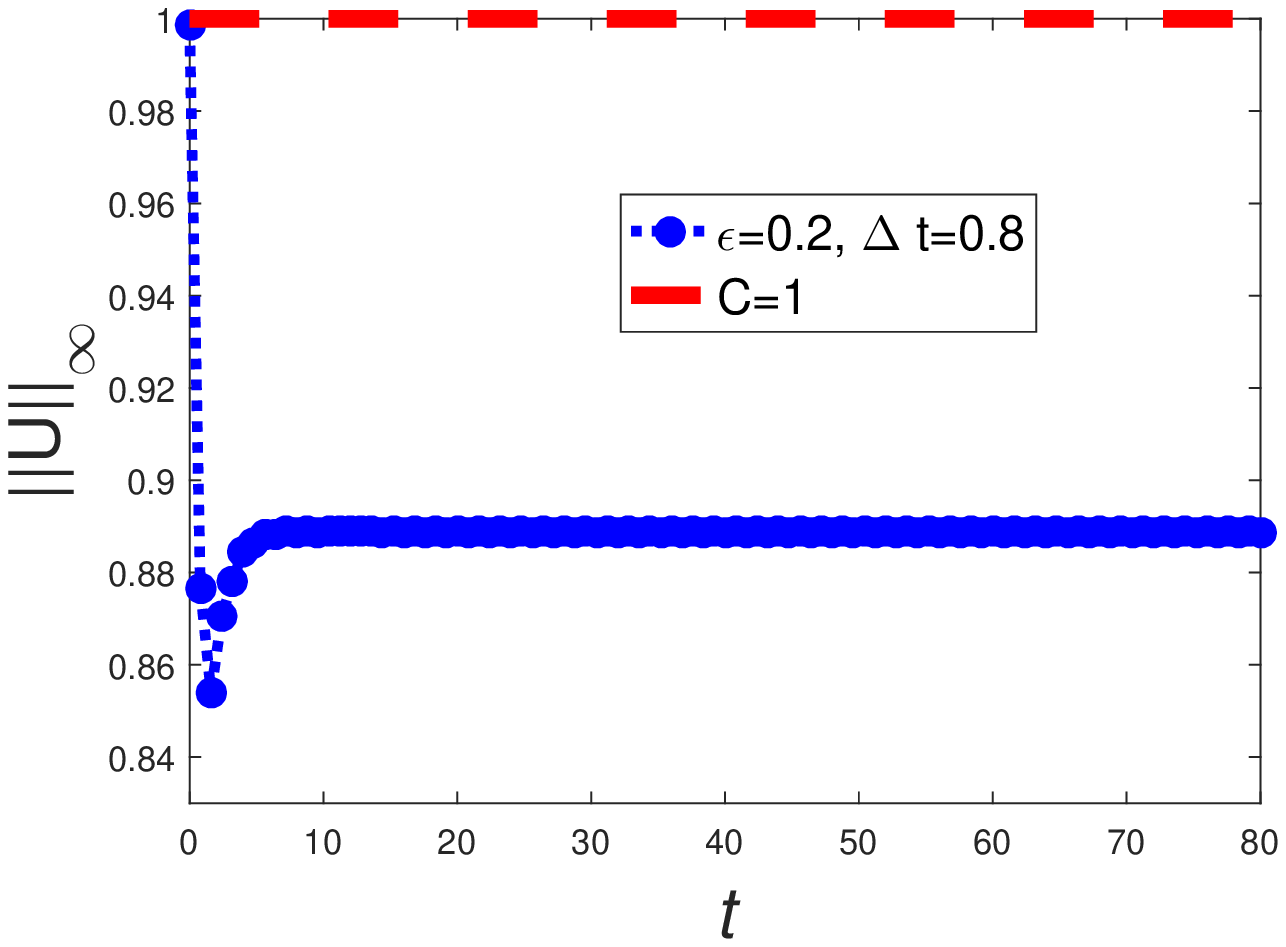}}
{
\includegraphics[width=0.23\textwidth]{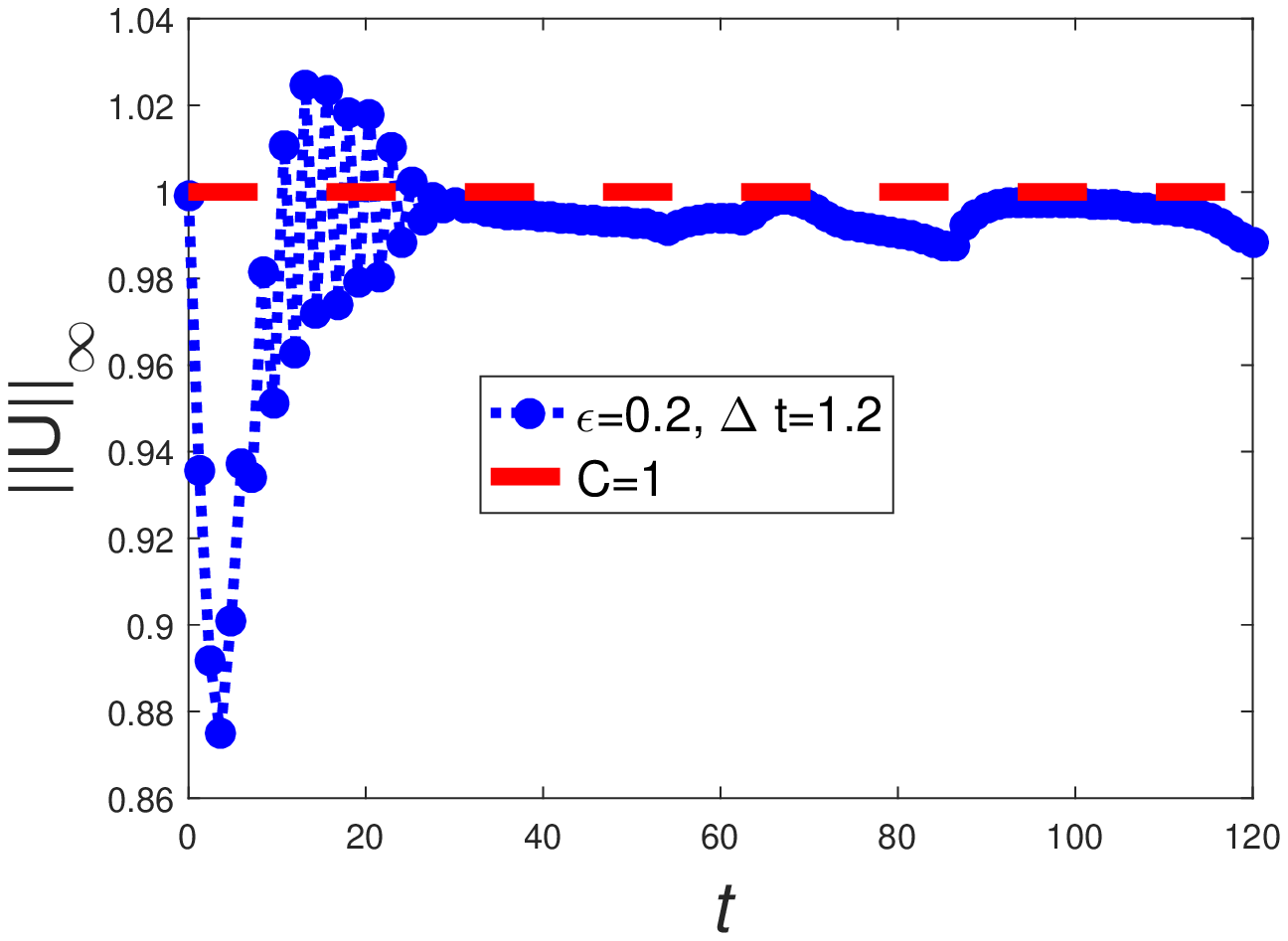}} \\
\caption{Numerical results with $\alpha=1.6$ and $h=0.05$: the maximum values of solution with different $\Delta t$ and $\varepsilon$.}
\label{Fig2}
\end{figure}

\begin{example}\label{ex4}
In this example, we consider the 2D  space fractional Allen-Cahn equation with initial condition
\begin{align}
u_0(x,y)=0.1\times rand(x,y)-0.05,
\end{align}
where zero boundary values are set for the initial condition $u_0(x,y)$.
\end{example}

\begin{figure}[!tbp]
\centering
{\includegraphics[width=0.32\textwidth]{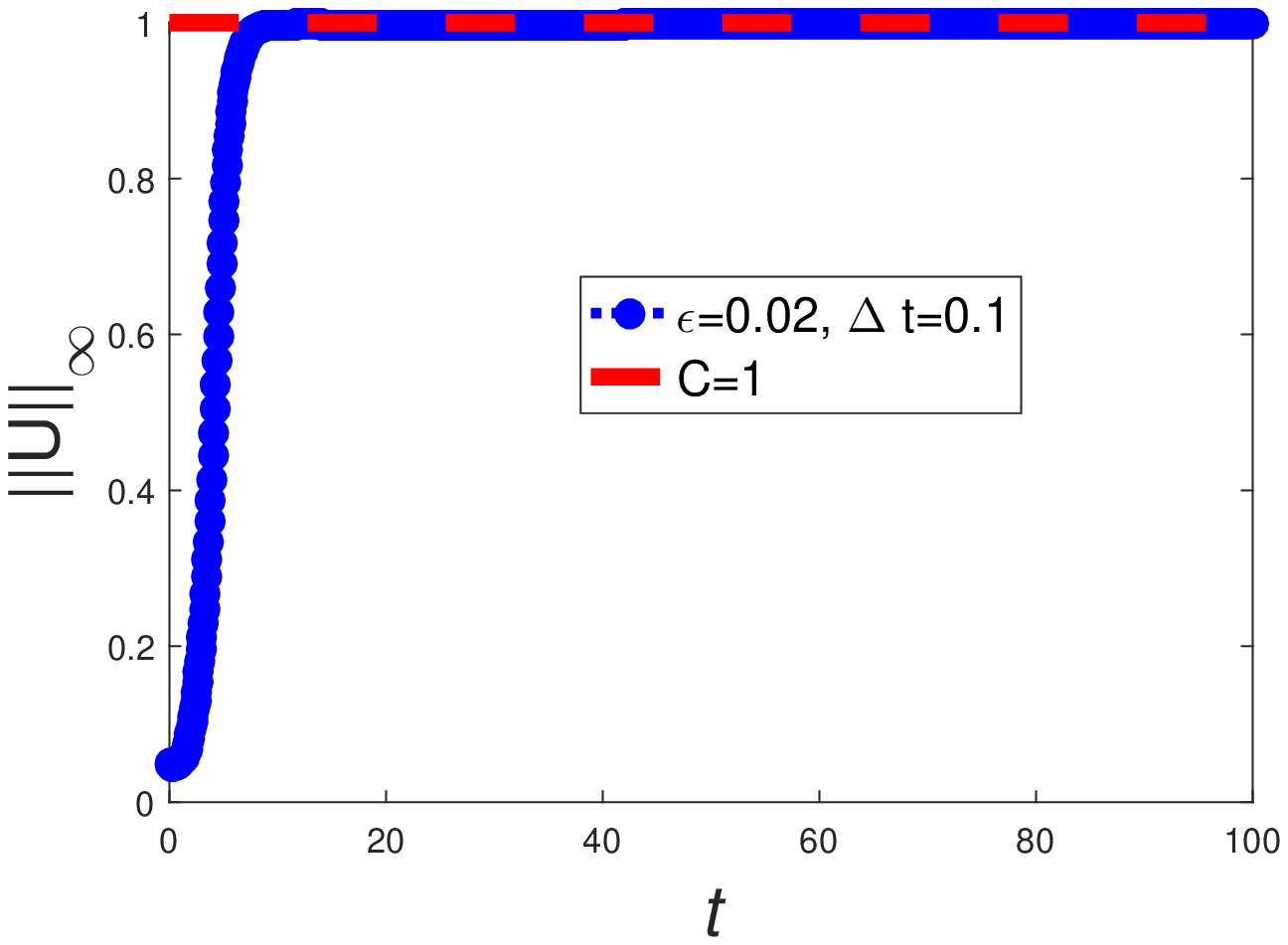}}
{\includegraphics[width=0.32\textwidth]{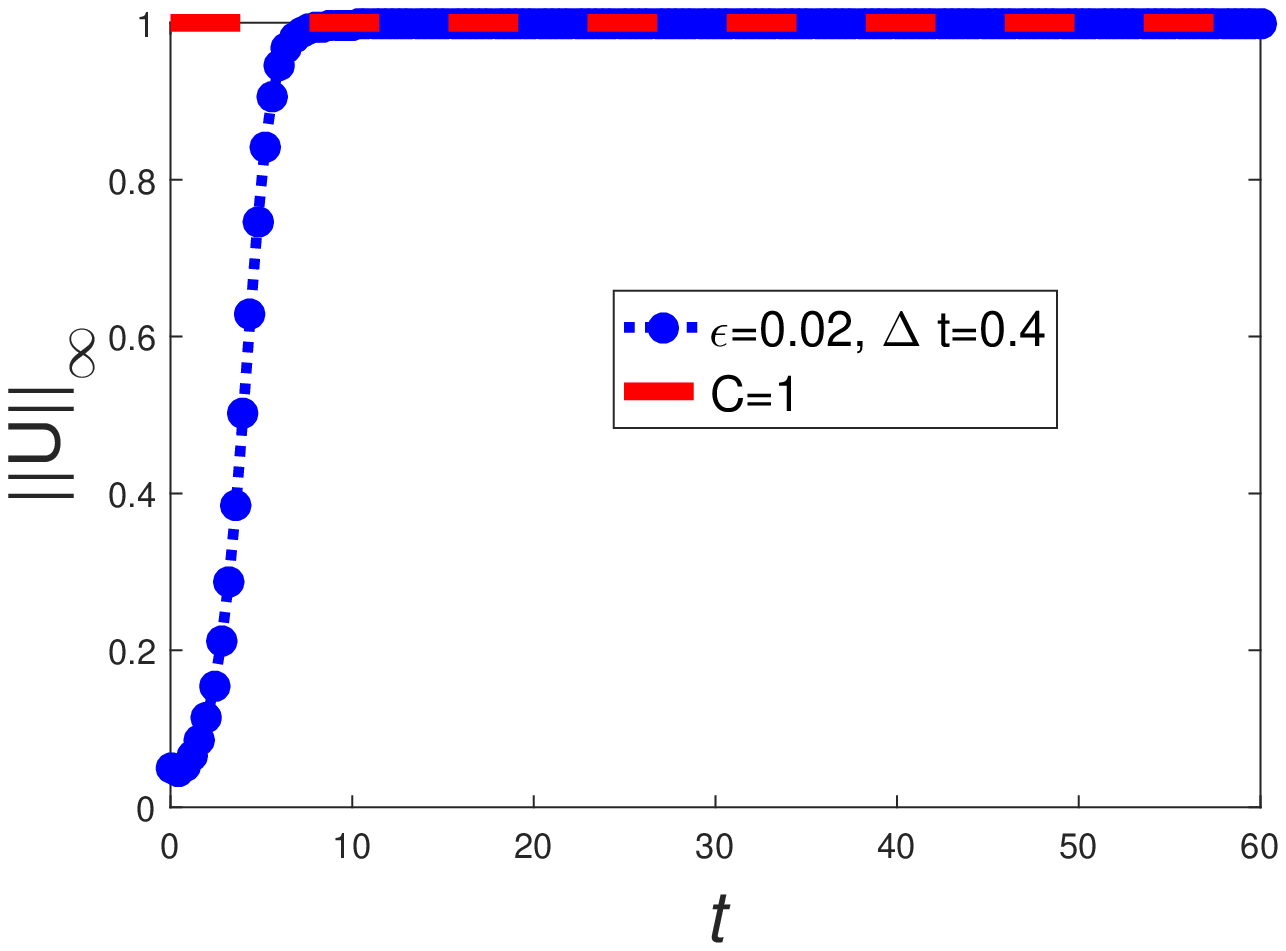}}
{\includegraphics[width=0.32\textwidth]{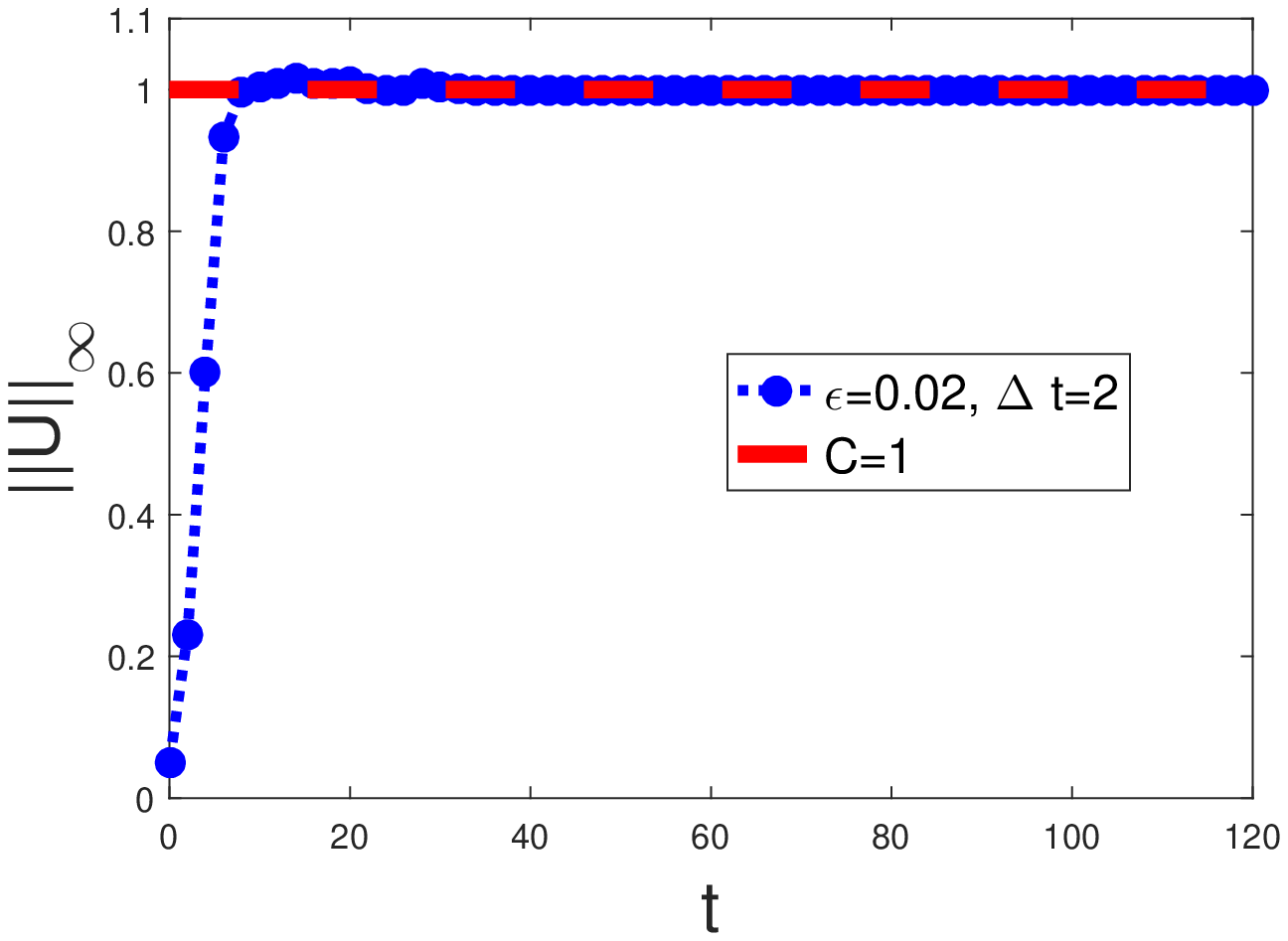}}
\caption{Numerical results with $\alpha=1.7, \varepsilon=0.02, h=0.01$: the maximum values of solution with different $\Delta t$.}
\label{Fig3}
\end{figure}

In this example, we first fix $h=0.01, \alpha=1.7$ and $\varepsilon=0.02$ but vary $\Delta t$. The maximum principle condition (\ref{condition}) requires $0.1776\leq\Delta t \leq 0.4945$.  Figure~\ref{Fig3} shows the maximum values of the numerical solutions are bounded by $1$ when $\Delta t=0.01$ and $\Delta t=0.4$. However, the maximum value exceeds 1 when $\Delta t$ increases to $2$.

Now we investigate the effects of fractional diffusion on phase separation and coarsening process. We set $h=0.01, \varepsilon=0.02, \Delta t=0.5$ and $\alpha=1.2, 1.5, 1.8$. Starting from random  initial values, the snapshots of the contours for the numerical solutions at $t=5,20,40,80$ are shown in Figure~\ref{Fig4}.  We see that reducing the fractional order yields to a thinner interfaces that allows smaller bulk regions and a much more heterogeneous phase structure. Moreover, it becomes slower for the phase coarsen process when the fractional order becomes smaller.

\begin{figure}[!tbp]
\centering
\includegraphics[scale=1]{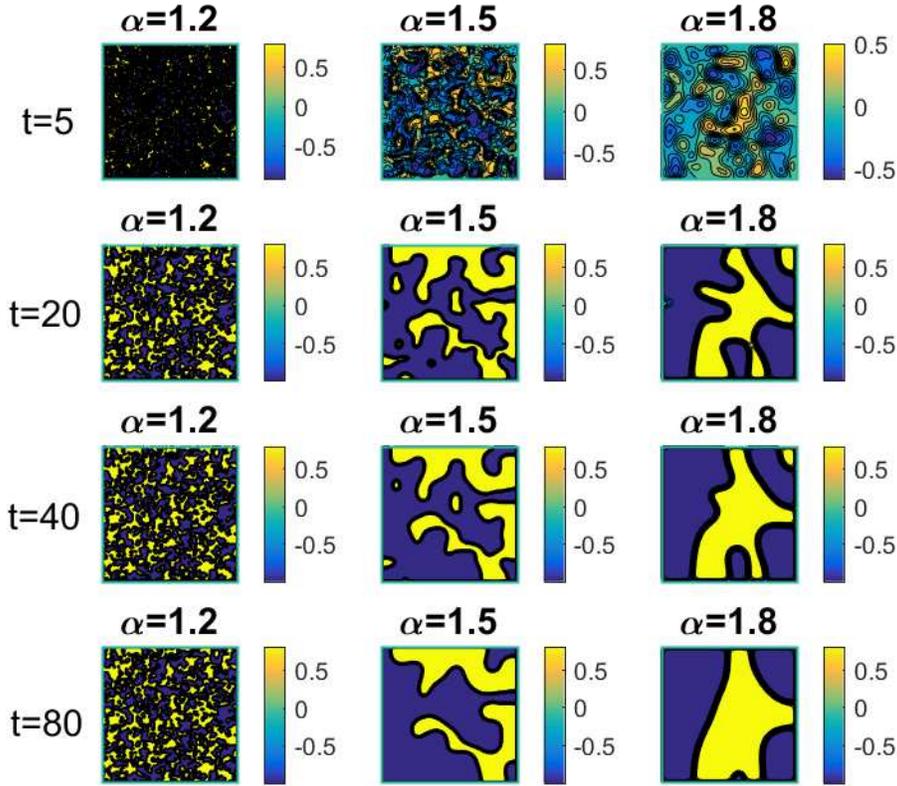}
\caption{Numerical dynamics (contour plots) with different fractional derivatives:  $\alpha=1.2, 1.5, 1.8$, where  $h=0.01, \varepsilon=0.02, \Delta t=0.5$.}
\label{Fig4}
\end{figure}

\begin{example}\label{ex5}
In this example, we consider the 3D  space fractional Allen-Cahn equation with exact solution
\begin{align}
u_0(x,y,z)=0.1\times rand(x,y,z)-0.05,
\end{align}
where zero boundary values  are set for the initial condition $u_0(x,y,z)$.
\end{example}

Again in this example, we first fix $h=0.01, \alpha=1.7, \varepsilon=0.02$ but vary $\Delta t$. The maximum principle condition (\ref{condition}) also gives $0.1776\leq\Delta t \leq 0.4945$ since the condition (\ref{condition}) does not rely on the dimension of the problem.  Same as the 2D case, Figure~\ref{Fig5} shows the maximum values of the numerical solutions are bounded by $1$ when $\Delta t=0.1$ and $\Delta t=0.4$. However, discrete maximum principle is invalid when $\Delta t$ increases to $2$.

\begin{figure}[!tbp]
\centering
{\includegraphics[width=0.32\textwidth]{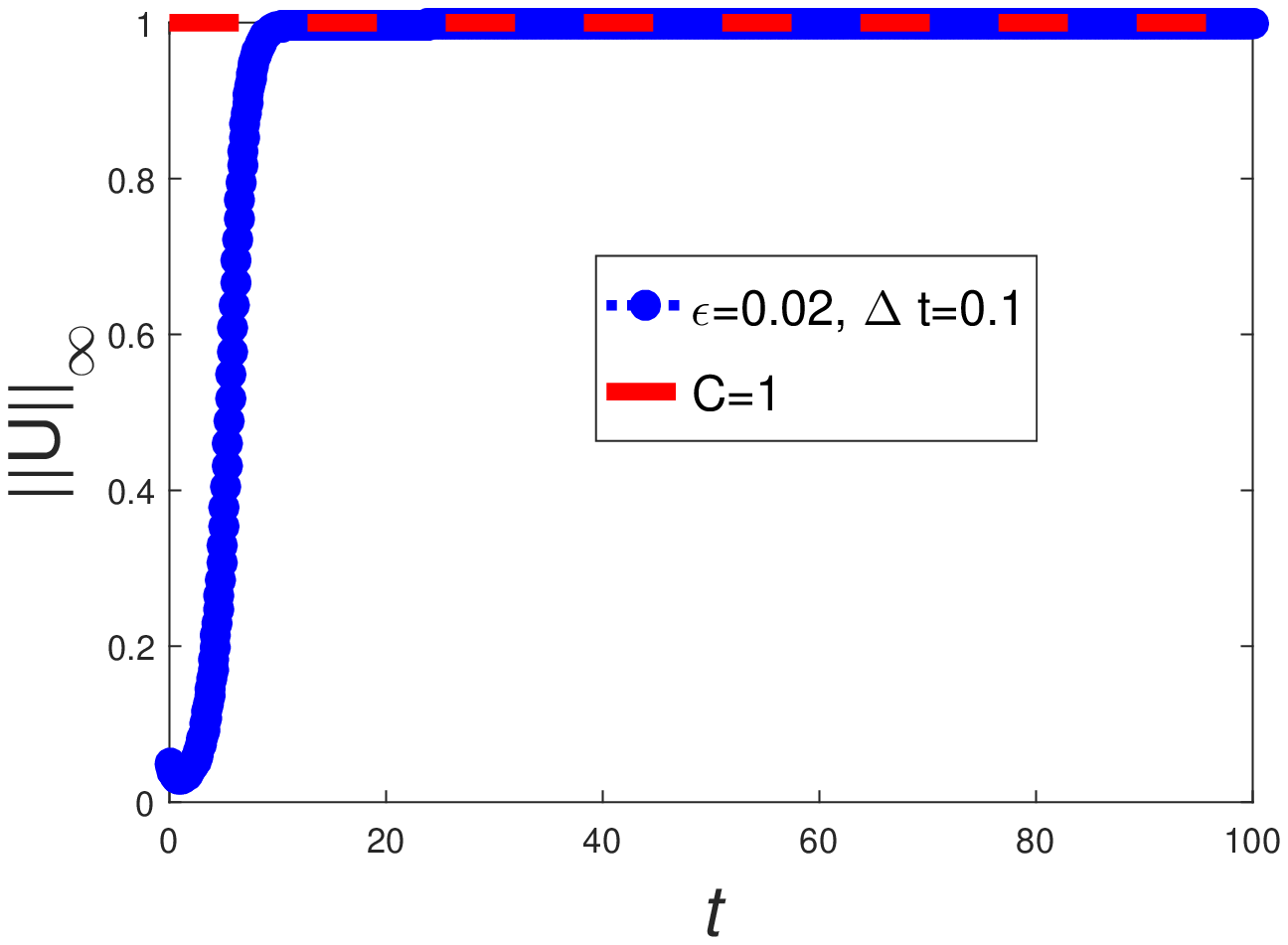}}
{\includegraphics[width=0.32\textwidth]{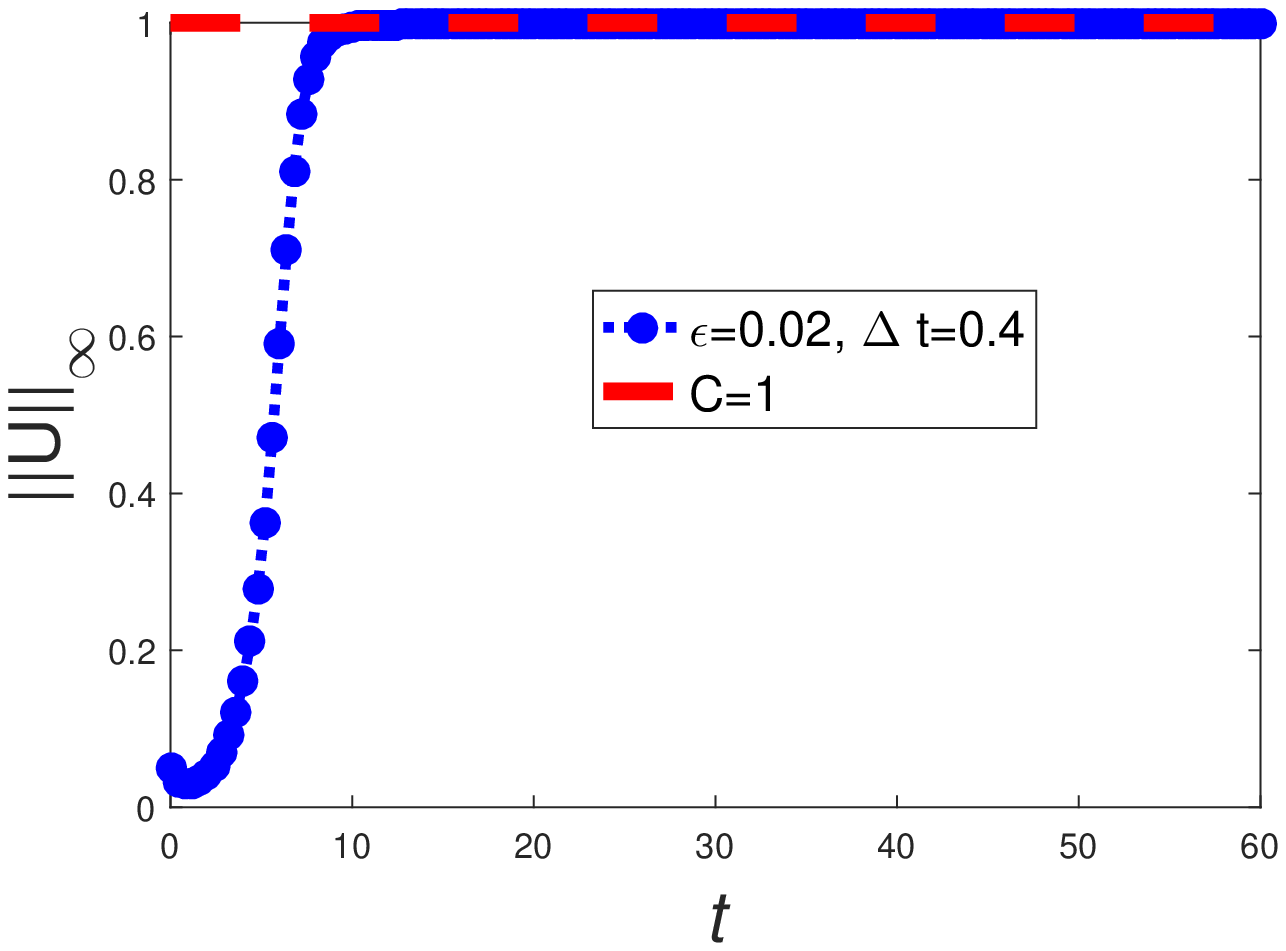}}
{\includegraphics[width=0.32\textwidth]{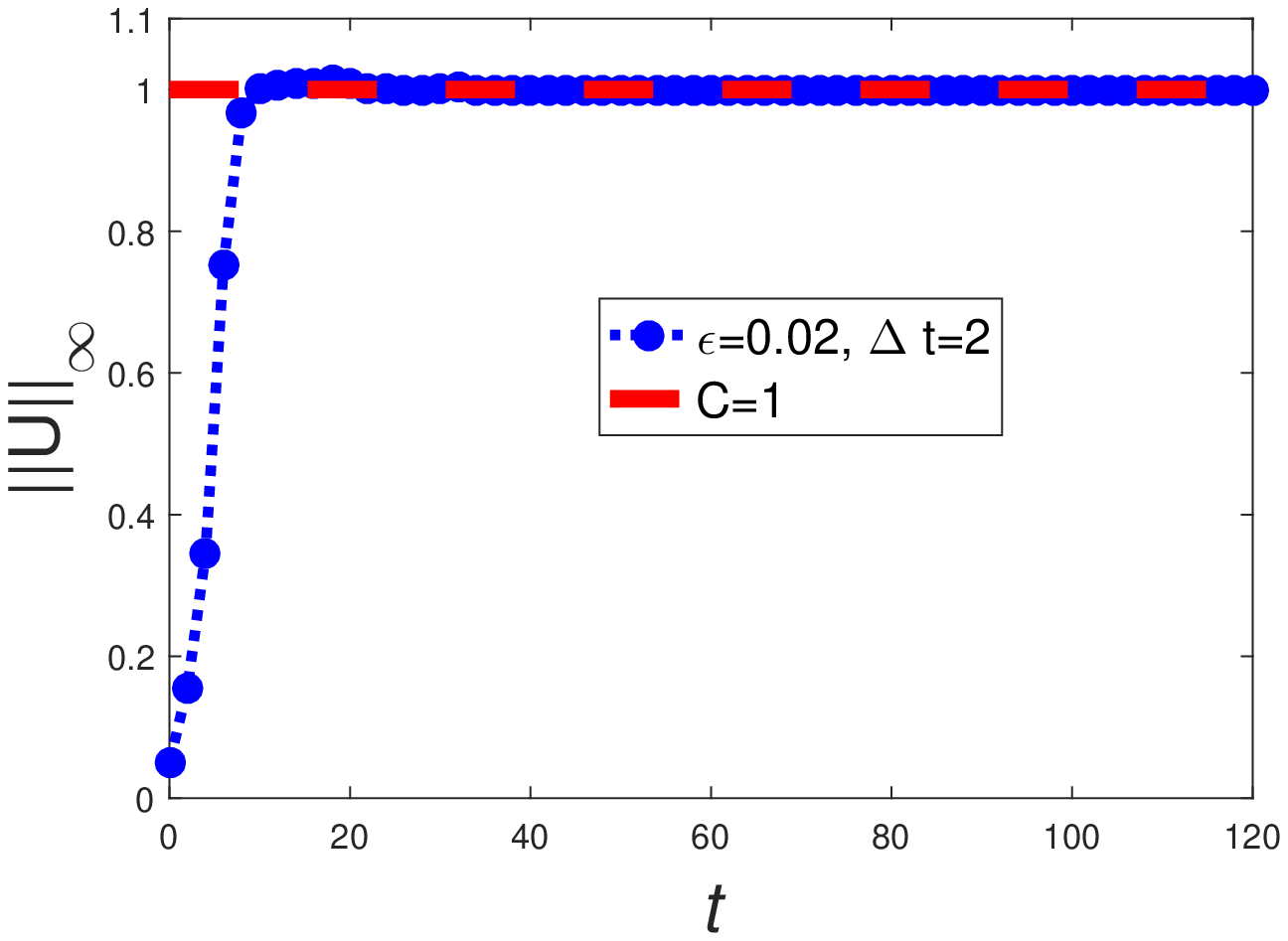}}
\caption{Numerical results with $\alpha=1.7, \varepsilon=0.02$ and $h=0.01$: the maximum values of solution with different $\Delta t$.}
\label{Fig5}
\end{figure}

Finally, we also investigate the effects of fractional diffusion on phase separation and coarsening process. We set $h=0.01, \varepsilon=0.02, \Delta t=0.5$ and $\alpha=1.2, 1.5, 1.8$. Starting from random  initial values, the snapshots of the contours for the numerical solutions at $t=5,20,40,80$ on the plane $z=0.5$ are shown in Figure~\ref{Fig6}.  Again, we see that reducing the fractional order yields to a thinner interfaces and it becomes slower for the phase coarsen process when the fractional order becomes smaller.

\begin{figure}[!tbp]
\centering
\includegraphics[scale=1.0]{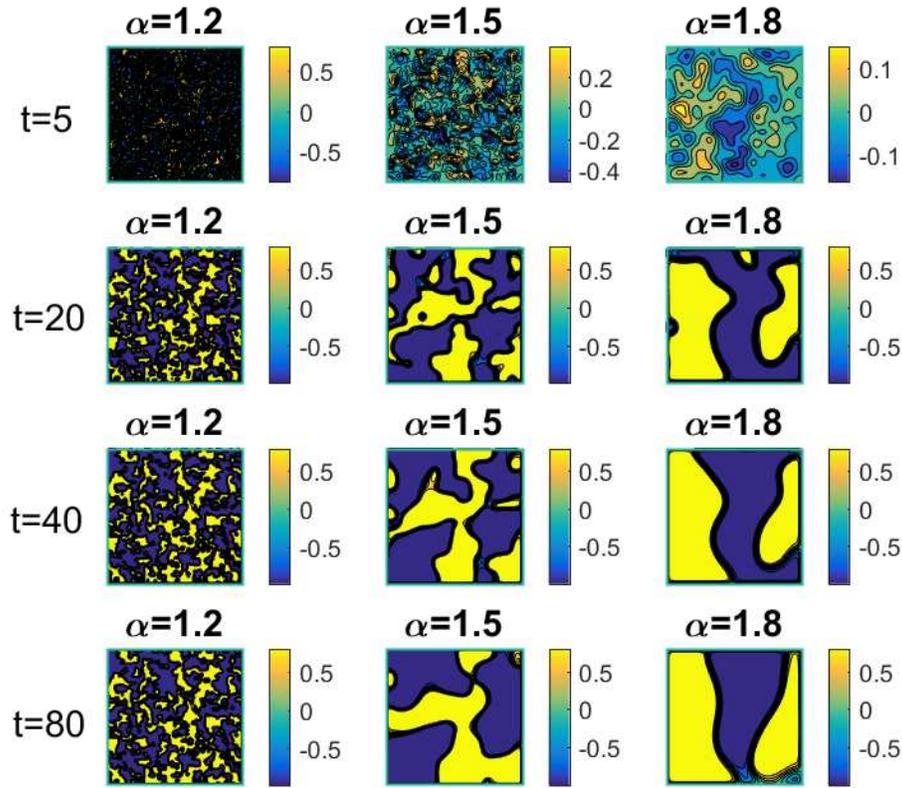}
\caption{Numerical dynamics  (contour plots on the plane $z=0.5$) with different fractional derivatives:  $\alpha=1.2, 1.5, 1.8$, where  $h=0.01, \varepsilon=0.02, \Delta t=0.5$.}
\label{Fig6}
\end{figure}

\section{Conclusions}\label{sec6}

In this paper, we developed a fourth-order maximum principle preserving operator splitting scheme for the space fractional Allen-Cahn equation. The second-order splitting method for the fractional  Allen-Cahn  equation splits the numerical procedure into three steps. The first and third steps involves  an ordinary differential equation that can be solved analytically. The intermediate step involves a linear multidimensional space fractional diffusion equation, which is solved by the ADI method  and fourth-order finite difference method.  A simple analysis for first and third steps together with a Fourier analysis for second ADI step show that the proposed operator splitting method is unconditionally stable for smooth solutions.  Additionally, under certain reasonable  time step constraint,   the discrete maximum principle is obtained. Finally, Richardson extrapolation is exploited to increase the temporal accuracy to fourth order. Numerical tests for both 2D and 3D space fractional Allen-Cahn  equations are carried out, for fabricated smooth solutions,  results show that the method is unconditionally stable and fourth-order accurate in both time and space variables.  More importantly, the discrete  maximum principle are numerically well verified.

The proposed linear scheme in this paper is fourth-order accurate, maximum principle preserving and unconditionally stable.  However, discrete energy decay law is generally not satisfied.  It will be very interesting to develop high-order linearized schemes with both discrete maximum principle and energy decreasing property. This will be our future objective.

\end{document}